\documentclass{amsart}

\usepackage{amsfonts, amsmath, amssymb}

\newtheorem{theorem}{Theorem}[section]
\newtheorem{lemma}[theorem]{Lemma}
\newtheorem{corollary}[theorem]{Corollary}
\newtheorem{fact}[theorem]{Fact}
\newtheorem{proposition}[theorem]{Proposition}

\theoremstyle{definition}
\newtheorem{example}[theorem]{Example}
\newtheorem{remark}[theorem]{Remark}
\newtheorem{definition}[theorem]{Definition}

\numberwithin{equation}{section}

\DeclareMathOperator{\Aut}{Aut}
\DeclareMathOperator{\Hom}{Hom}
\DeclareMathOperator{\dcl}{dcl}
\DeclareMathOperator{\tp}{tp}
\DeclareMathOperator{\id}{id}
\DeclareMathOperator{\eq}{eq}
\DeclareMathOperator{\Th}{Th}
\DeclareMathOperator{\dom}{dom}
\DeclareMathOperator{\ran}{ran}

\usepackage{tikz}
\usetikzlibrary{cd}

\title[Groupoid-cover correspondence]{Functoriality and uniformity in Hrushovski's groupoid-cover correspondence}

\author{Levon Haykazyan}
\address{Levon Haykazyan\\
University of Waterloo\\
Department of Pure Mathematics\\
200 University Avenue West\\
Waterloo, Ontario \  N2L 3G1\\
Canada}
\email{lhaykazyan@uwaterloo.ca}

\author{Rahim Moosa}
\address{Rahim Moosa\\
University of Waterloo\\
Department of Pure Mathematics\\
200 University Avenue West\\
Waterloo, Ontario \  N2L 3G1\\
Canada}
\email{rmoosa@uwaterloo.ca}

\date{\today}
\begin{document}

\begin{abstract}
The correspondence between definable connected groupoids in a theory $T$ and internal generalised imaginary sorts of $T$, established by Hrushovski in [``Groupoids, imaginaries and internal covers," {\em Turkish Journal of Mathematics}, 2012], is here extended in two ways:
First, it is shown that the correspondence is in fact an equivalence of categories, with respect to appropriate notions of morphism.
Secondly, the equivalence of categories is shown to vary uniformly in definable families, with respect to an appropriate relativisation of these categories.
Some elaboration on Hrushovki's original constructions are also included.
\end{abstract}

\maketitle

\setcounter{tocdepth}{1}
\tableofcontents

\section {Introduction}

\noindent
The notion of a definable set $X$ being {\em internal} to another definable set $Y$ has been of central significance in geometric stability theory, as part of the study of the fine structure of definable sets.
Roughly speaking, it means that $X$ is definably isomorphic to a definable set in $Y^{\eq}$, though it is crucial that the isomorphism may involve more parameters than those required to define $X$ and $Y$.
For example, in differentially closed fields ($\operatorname{DCF}_0$), we can take $Y$ to be the constant field and $X$ to be the set of solutions to the algebraic differential equation $x'=1$.
Then $X$ is definably isomorphic to $Y$ after we name any one solution to the equation.
This dependence on additional parameters is controlled by the {\em liason} or {\em binding group}, an interpretable group acting definably on $X$ and agreeing with the action of the group of automorphisms of the universe that fix $Y$ pointwise.
The existence and importance of the liason group was recognised already by Zilber~\cite{zilber} in the late nineteen seventies.
Poizat~\cite{poizat} realised that this theory of definable automorphism groups, when applied to $\operatorname{DCF}_0$, could be used to recover Kolchin's differential Galois theory.
Hrushovski gave the modern and full development of the subject, first in the context of stable theories and later in complete generality; in~\cite{udigeneral} the liason group is constructed from internality assuming only that $Y$ is {\em stably embedded}, i.e., that the full induced structure agrees with the $0$-definable induced structure on $Y$.

In~\cite{hrushovski}, Hrushovski addresses the question of how to describe $X$ from the point of view of the induced structure on $Y$.
The liason group lives naturally and without additional parameters in $(X, Y)^{\eq}$, but does not in general have a canonical manifestation in $Y^{\eq}$.
Hrushovski introduces  {\em liason groupoids}, establishing a bijective correspondence (up to a natural notion of equivalence) between definable connected groupoids in a theory $T$ and {\em internal covers} of $T$; i.e., expansions of $T$ that are internal to $T$ and in which $T$ is stably embedded.
In the above notation, $T$ is the theory of $Y^{\eq}$ and the internal cover is the theory of $(X,Y)^{\eq}$.
So, what is effected by Hrushovski's correspondence is an encoding of internal covers of $T$ as definable connected groupoids in $T$, in very much the same way that the sorts of $T^{\operatorname{eq}}$ are encoded in $T$ as definable equivalence relations.

In fact there is a small mistake in~\cite{hrushovski} so that the correspondence does not encompass all definable connected groupoids but only those satisfying the additional technical condition of {\em finite faithfulness}.\footnote{This mistake in~\cite[Theorem~3.2]{hrushovski} does not invalidate any of the other theorems of that paper, as they relate to contexts where finite faithfulness is automatic; namely when the theory is stable or when the automorphism groups in the groupoid are finite.}
This is explained and corrected in Section~\ref{section-covers-gpd} below,  where we also present a brief description of Hrushovski's constructions, elaborating on a few aspects.

This paper extends Hrushovski's work in two ways.
First, we show that his correspondence comes from an equivalence of categories.
There is a more or less natural way to make internal covers into a category: a morphism between two internal covers should be a function that is definable in a common cover.
This is made precise in~$\S$\ref{subsectcovmor}, but for example, a definable function between definable sets in $\operatorname{DCF}_0$ that are internal to the constants will give rise to a morphism between the corresponding internal covers of $\operatorname{ACF}_0$.
With this notion of morphism, isomorphic internal covers are precisely those that are equivalent in Hrushovski's sense.
On the other hand, it is not so clear what we should be regarding as morphisms between definable groupoids.
The answer, given in~$\S$\ref{subsectgpdmor}, takes a somewhat unexpected form.
In any case, we are able to show, in the rest of Section~\ref{section-functorial}, that Hrushovski's correspondence is an equivalence of these categories.

Secondly, we pursue a suggestion in~\cite{hrushovski} and address the question of uniformity in definable families.
Allowing connected groupoids to vary in a definable family amounts to dropping the assumption of connectedness, and allowing internal covers to vary definably amounts to replacing internality with {\em $1$-analysability}.
We prove that Hrushovski's constructions -- of the liason groupoid associated to an internal cover and of the internal cover associated to a finitely faithful definable connected groupoid -- do vary uniformly with the parameters, but only under the (necessary) additional assumption that there is no significant interaction amongst the varying internal covers (called {\em independence of fibres} below).
This is done in Section~\ref{section-uniformity}.

\bigskip
\section{Preliminaries on covers}
\label{coversection}
\noindent
We review in this brief section the notion of a cover of a theory.
Details on the various model-theoretic notions that come up can be found in~\cite{tentziegler}.

Suppose $M$ is a (multi-sorted) structure.
Throughout this paper, {\em definable} means definable with parameters, unless stated otherwise.
A $0$-definable set $X$ is said to be {\em stably embedded in $M$} if for every $n>0$, every definable subset of $X^n$ is definable with parameters from $X$.
A {\em stably embedded reduct} of $M$ is a reduct $N$ with the property that every subset of a finite cartesian product of sorts in $N$ that is definable in $M$ is definable in $N$.
So $X$ is stably embedded in $M$ if and only if the structure induced on $X$ by the $0$-definable sets of $M$ is a stably embedded reduct of $M$.
Here are some other useful characterisations; see the Appendix of~\cite{zoe-udi} for details.

\begin{fact}
\label{seset}
Suppose $M$ is sufficiently saturated and $X$ is a $0$-definable set in $M$.
The following are equivalent:
\begin{itemize}
\item[(i)]
$X$ is stably embedded.
\item[(ii)]
For any tuple $a$ of small length from $M$ there is a small subset $B\subseteq X$ such that $\tp(a/B)\vdash\tp(a/X)$.
\item[(iii)]
If $a$ and $b$ are tuples from $M$ of small length with $\tp(a/X)=\tp(b/X)$ then there exists $\sigma\in\Aut(M)$ that fixes $X$ pointwise and such that $\sigma(a)=b$.
\end{itemize}
\end{fact}

\begin{fact}
\label{sereduct}
Suppose $M$ is sufficiently saturated and $N$ is a reduct of $M$.
The following are equivalent:
\begin{itemize}
\item[(i)]
$N$ is stably embedded.
\item[(ii)]
The restriction map $\Aut(M)\to\Aut(N)$ is surjective.
We denote the kernel of this map by $\Aut(M/N)$.
\end{itemize}
\end{fact}

That we have worked with $0$-definable sets is not so important.
If $A\subset M$ is a small set of parameters, and $X$ is $A$-definable, then $X$ is said to be {\em stably embedded over $A$} if it is stably embedded in $M_A$, the structure obtained by naming the elements of $A$ as new constants.

\begin{definition}[Hrushovski~\cite{hrushovski}]
Let $T$ be a complete multi-sorted first order theory with elimination of
imaginaries in a language $L$.
Consider an expansion $L'$ of $L$.
Let $T'$ be a complete $L'$-theory that extends $T$.
We say that $T'$ is a {\em cover} of~$T$ if in every model of~$T'$ the reduct to $L$ is stably embedded.
\end{definition}

Note that if $L'$ does not involve new sorts then any cover of $T$ in $L'$ is a definitional expansion of $T$.
So we are really interested in expansions involving new sorts.

As is the tradition in model theory, we fix a monster model $\mathbb U$ of $T$
and work inside $\mathbb U$. Accordingly we call an $L'$-expansion $\mathbb U'$
a {\em cover} of $\mathbb U$ if it is saturated and $\mathbb U$ is stably
embedded in $\mathbb U'$.
This is an abuse of notation as two such expansions $\mathbb U_1$ and $\mathbb U_2$ of $\mathbb U$ may be elementarily equivalent and therefore determine the same cover of $T$.
But by saturation this would mean there is an isomorphism $\sigma:\mathbb U_1\to\mathbb U_2$.
Stable embeddability implies,  by Fact~\ref{sereduct}(ii), that $\sigma|_{\mathbb U}$ extends to an automorphism $\alpha$ of $\mathbb U_2$, and so $\alpha^{-1}\sigma$ yields an isomorphism that is the identity on $\mathbb U$.
That is,  $\mathbb U_1$ and $\mathbb U_2$  give the same cover of $T$ if and only if they are isomorphic over $\mathbb U$.

In later sections we will want to consider unions of covers.
Suppose $\mathbb U_1=(\mathbb U,(S_{1j})_{j\in J_1},\dots)$ and $\mathbb U_2=(\mathbb U,(S_{2j})_{j\in J_2},\dots)$ are covers of $\mathbb U$, where $(S_{ij})_{j\in J_i}$ are the new sorts of $\mathbb U_i$ for $i=1,2$.
Then  $\mathbb U_1\cup\mathbb U_2$ is by definition the expansion of $\mathbb U$ that has new sorts $(S_{ij}:i=1,2, j\in J_i)$ and whose structure is as dictated by $\mathbb U_1$ and $\mathbb U_2$ on the appropriate sorts.
The language in which this structure is considered is $L$ together with the disjoint union of $L_1\setminus L$ and $L_2\setminus L$. 

\begin{lemma}
\label{union}
If $\mathbb U_1$ and $\mathbb U_2$ are covers of $\mathbb U$ then $\mathbb U_1\cup\mathbb U_2$ is a cover of $\mathbb U_1$ and $\mathbb U_2$, and $\operatorname{Th}(\mathbb U_1\cup\mathbb U_2)=\operatorname{Th}(\mathbb U_1)\cup \operatorname{Th}(\mathbb U_2)$.
\end{lemma}

\begin{proof}
A saturated (in the same cardinality as $\mathbb U_1,\mathbb U_2$) model of $\operatorname{Th}(\mathbb U_1)\cup \operatorname{Th}(\mathbb U_2)$ will be of the form $\mathbb V_1\cup\mathbb V_2$ with $\mathbb V_i$ a saturated model of $\operatorname{Th}(\mathbb U_i)$.
We therefore have isomorphisms $\sigma_i:\mathbb U_i\to\mathbb V_i$.
Now $(\sigma_2|_{\mathbb U})^{-1}\sigma_1|_{\mathbb U}\in\Aut(\mathbb U)$, and as $\mathbb U$ is stably embedded in $\mathbb U_2$ we can extend it to an automorphism $\alpha$ of $\mathbb U_2$ by Fact~\ref{sereduct}(ii).
Replacing $\sigma_2$ by $\sigma_2\alpha$, we may therefore assume that $\sigma_1$ and $\sigma_2$ agree on $\mathbb U$.
So $\sigma_1\cup\sigma_2$ is an isomorphism from $\mathbb U_1\cup\mathbb U_2$ to $\mathbb V_1\cup\mathbb V_2$.
This shows that $\mathbb U_1\cup\mathbb U_2$ is saturated and $\operatorname{Th}(\mathbb U_1\cup\mathbb U_2)=\operatorname{Th}(\mathbb U_1)\cup \operatorname{Th}(\mathbb U_2)$.

It remains to show that each $\mathbb U_i$ is stably embedded in $\mathbb U_1\cup\mathbb U_2$.
Suppose $\sigma\in\Aut(\mathbb U_1)$.
Then $\sigma|_{\mathbb U}\in\Aut(\mathbb U)$ and hence extends to an automorphism $\sigma'$ of $\mathbb U_2$.
Then $\sigma\cup\sigma'\in\Aut(\mathbb U_1\cup\mathbb U_2).$
A similar argument shows that every automorphism of $\mathbb U_2$ extends to $\mathbb U_1\cup\mathbb U_2$.
\end{proof}

\bigskip
\section{Internal covers and definable groupoids}
\label{section-covers-gpd}

\noindent
This section is primarily a survey of the correspondence that Hrushovski established in the first few sections of~\cite{hrushovski} between definable connected groupoids and internal covers.
We take this opportunity, however, to correct an error in the published statement and proof of that correspondence.
We have also streamlined some of the arguments and elaborated on some of the ideas.
Our presentation is largely self-contained, and we have made an effort to be explicit about where we have diverged from Hrushovski's notations and conventions.

We will focus on covers which involve only a finite increase in the language.

\begin{definition}
A cover $T'$ of $T$ is said to be in a language that is {\em finitely generated} over $L$ if there exists a sublanguage $L'_\circ\subseteq L'$ such that $L_\circ'\setminus L$ is finite and every $L'$-formula is equivalent modulo $T'$ to an $L'_\circ$-formula.
\end{definition}

We will be interested in {\em internal} covers $\mathbb U'$ of $\mathbb U$.
Since $T$ eliminates imaginaries this can be characterised by saying that each new sort $S$ of $\mathbb U'$ is definably isomorphic (over possibly additional parameters) with a definable set in $\mathbb U$.

Our primary focus is on internal covers that involve only one new sort.

\begin{definition}[Hrushovski~\cite{hrushovski}]
An {\em internal generalised imaginary sort} of $\mathbb U$ is an internal cover $\mathbb U'$ with only one new sort $S$, such that the language of $\mathbb U'$ is finitely generated over the language of $\mathbb U$.
\end{definition}

There is a canonical example of such internal covers.

\begin{example}
Let $\mathbb U=(G,\cdot,e)$ be a group and $S$ a principal homogeneous space for $G$ with the action denoted by $\rho:G\times S\to S$.
Then $\mathbb U' = (G,S,\cdot,e,\rho)$ is an internal cover with new sort $S$ witnessed by fixing any $a\in S$ and considering the definable bijection $\rho_a:G\to S$ given by $g\mapsto\rho(g,a)$.

If we fix $a\in S$ then we have another, $a$-definable, action of $G$ on $S$: each $g\in G$ gives rise to the definable bijection $\alpha_g:S\to S$ given by $\alpha_g(s)= \rho_a((\rho_a)^{-1}(s)g^{-1})$.
Note that $\alpha_g \cup \id_G$ preserves $\rho$ and so is an automorphism of $\mathbb U'$.
The correspondence $g \mapsto \alpha_g\cup \id_G$ is an isomorphism between $G$ and $\Aut(\mathbb U'/\mathbb U)$, as groups acting on $S$.
\end{example}

As is well known, this phenomenon occurs under very broad conditions: if $\mathbb U'$ is an internal generalised imaginary sort then $\Aut(\mathbb U'/\mathbb U)$ together with its action on $S$ is interpretable in $\mathbb U'$.
It is called the liason group.
However, in general we need parameters to realise this group in $\mathbb U$.
(In the above example parameters were not needed to define the group, just to define the action).
The parameter-free construction of Hrushovski~\cite{hrushovski} gives instead a definable liason {\em groupoid} in $\mathbb U$.
Moreover, this groupoid captures precisely the internal cover $\mathbb U'$.

A groupoid $\mathcal G$ is a category where all morphisms are invertible. It is
called {\em connected} if there is a morphism between any two objects. If $a, b$
are objects in $\mathcal G$, then we denote by $\Hom_{\mathcal G}(a, b)$ the set
of morphisms from $a$ to $b$ and also denote $\Aut_{\mathcal G}(a) =
\Hom_{\mathcal G}(a, a)$. If $\mathcal G$ is clear, we may skip it in the subscript.
Note that if $\mathcal G$ is connected then all these automorphism groups are isomorphic.

A {\em definable groupoid} consists of a definable family $(O_i)_{i \in I}$ of definable sets and
a definable family $(f_m)_{m \in M}$ of definable bijections between them, closed under composition and taking inverses.
It is worth being more explicit here.
There are definable sets~$I$, which encodes the set of objects, and surjective definable $M\to I\times I$, which encodes the set of morphisms.
There is a definable set $O$ and a surjective definable map $\pi : O \to I$
such that $O_i = \pi^{-1}(i)$, and a definable subset $f \subseteq M\times_{I\times I}(O \times O)$ such that $f_m = \{(x, y) \in O^2 : (m,x, y) \in f\}$ is the graph of a bijection
between $O_i$ and $O_j$ where $m$ lives over $(i,j)$.  Further we require that
distinct $m \in M$ define distinct bijections. The groupoid structure on $(I,
M)$ (i.e. the composition of two morphism, or the domain or the range of a
morphism) is then definable from this data.
We denote by $M(i,j)$ the fibre of $M\to I\times I$ above $(i,j)$, so that $(f_m)_{m\in M(i,j)}$ are the morphisms from $O_i$ to $O_j$.

\begin{remark}
Our terminology is slightly different from that of Hrushovski~\cite{hrushovski}. There the
structure $(I, M)$ (together with functions of partial composition, domain and
range) is called a definable groupoid. Then he considers a faithful definable
functor from $(I, M)$ to the category of definable sets and definable functions.
This gives rise to $(O_i, f_m)_{i \in I, m \in M}$ which he called a {\em
concrete definable groupoid}.
So all our definable groupoids are concrete.
This has the slightly awkward consequence that whereas for Hrushovski a definable groupoid with one object is nothing more than a definable group, for us it is a definable group action.
\end{remark}

We present below Hrushovski's theorem~\cite[Proposition~2.5]{hrushovski} that associates a liason groupoid to an internal cover.
We include a proof here mostly for the sake of completeness; ours differs superficially from Hrushovski's in that we do not go via $*$-definable groupoids but rather use finite-generatedness of the language to construct directly a definable groupoid.

\begin{theorem}[Hrushovski]
\label{groupoid-construction}
Let $\mathbb U' = (\mathbb U, S,\dots)$ be an internal generalised imaginary sort of~$\mathbb U$.
There is a $0$-definable
connected groupoid $\mathcal G'$ in $(\mathbb U')^{\eq}$ such that:
\begin{itemize}
\item The objects of $\mathcal G'$ are definable sets in $\mathbb U$ together
with $S$.\\
In particular, the restriction $\mathcal G$ of $\mathcal G'$ to $\mathbb U$ -- i.e., the full
subcategory generated by all objects excepts $S$ -- is a $0$-definable connected groupoid
in $\mathbb U$.
It is called the {\em liason groupoid} of $\mathbb U'$.
\item $\Aut_{\mathcal G'}(S)$, together with its action on $S$, is isomorphic to $\Aut(\mathbb U'/\mathbb U)$.
\end{itemize}
\end{theorem}

\begin{proof}
Since $\mathbb U$ eliminates imaginaries, there is a definable
subset $O_b$ in $\mathbb U$ and a definable bijection $f_c : S \to O_b$, where $b \in \mathbb U$ and $c \in \mathbb U'$ are the parameters used in the respective definitions.
Given $(b',c')$ we can consider the corresponding formulas with parameters $(b',c')$ in place of $(b,c)$; and when this again gives rise to a definable function we will write it as $f_{c'}:S\to O_{b'}$.

Since $\mathbb U$ is stably embedded, there is a small set $B\subseteq\mathbb U$  -- which we may assume contains $b$ -- such that
$\tp(c/B) \vdash \tp(c/\mathbb U)$. Let $p(x) = \tp(c/B)$. If $d, e \models p$,
then there is an automorphism $\sigma \in \Aut(\mathbb U'/\mathbb U)$ such that
$\sigma(e) = d$. Hence for any $a \in S$ we have
$$f_e(a) = \sigma(f_e(a)) = f_{\sigma(e)}(\sigma(a)) = f_d(\sigma(a)).$$
Thus $f_d^{-1} \circ f_e = \sigma|_S$ and so $f_d^{-1} \circ f_e \cup
\id_{\mathbb U}$ is elementary. But since the language of $\mathbb U'$ is finitely generated over that of $\mathbb U$, there is a formula $\theta(d, e)$ asserting this. Moreover, we have shown that $p(x)
\cup p(y) \vdash \theta(x, y)$. Hence by compactness there is a formula
$\phi(x, g) \in p(x)$ such that
$$\phi(x, g) \land \phi(y, g) \vdash \theta(x, y).$$
By extending $b$ and $g$ we may assume that $b = g$. Again by compactness there is a
formula $\psi(z) \in \tp(b)$ such that
$$\psi(z) \vdash \forall x, y (\phi(x, z) \land \phi(y, z) \to \theta(x, y))$$

Now as objects of $\mathcal G'$ we take $S$ together with $O_{b'}$ where $b'$
realises $\psi(z)$. The morphisms from $S$ to $O_{b'}$ are $f_{c'}$ where $c'$
realises $\phi(x, b')$. Note that if $f_{c'}, f_{c''} \in \Hom(S, O_{b'})$, then
$\theta(c', c'')$ holds and therefore $f_{c'}^{-1} \circ f_{c''} \cup \id_{\mathbb U}\in
\Aut(\mathbb U'/\mathbb U)$. Conversely if $\sigma \in \Aut(\mathbb U'/\mathbb U)$ then $f_{c'}
\circ \sigma|_S = f_{\sigma^{-1}(c')} \in \Hom(S, O_b')$.
Therefore, for any $b'$ realising $\psi(z)$, and for any $c'$
realising $\phi(x, b')$,
\begin{equation}
\label{homaut}
\Hom(S, O_{b'}) = f_{c'} \circ \Aut(\mathbb U'/\mathbb U).
\end{equation}
The rest of the groupoid can be defined from this. We set
$$\begin{array}{rcl}
\Hom(O_{b'}, S) & = & \{f^{-1} : f \in \Hom(S, O_{b'})\} \\
\Hom(O_{b'}, O_{b''}) & = & \{f \circ g^{-1} : f \in \Hom(S, O_{b''}), g \in
\Hom(S, O_{b'}) \}\\
\Hom(S, S) & = & \{f^{-1} \circ g : f,g \in \Hom(S, O_{b'}) \text { for some }
b'\}
\end{array}$$
It is not hard to verify that $\mathcal G'$ is indeed a $0$-definable connected groupoid in $\mathbb U'$.
Note also that from~(\ref{homaut}) it follows that $\Hom(S, S) = \Aut(\mathbb U'/\mathbb U)$.
\end{proof}

The following proposition shows that finite-generatedness of the language is a necessary hypothesis of the above the theorem.
It will also be of use later.

\begin{proposition}
\label{finite-language}
Let $\mathbb U'$ be an internal cover of $\mathbb U$ with one new sort $S$, such that $\Aut(\mathbb U'/\mathbb U)$ together with its action on $S$ is interpretable in $\mathbb U'$.
Then the language of $\mathbb U'$ is finitely generated over $\mathbb U$.
 \end{proposition}

\begin{proof}
An immediate consequence of the definability of the action of $\Aut(\mathbb U'/\mathbb U)$ on $S$ is that, working in $(\mathbb U')^{\eq}$, for any formula $\phi(x,y,z)$, where $x$ and $y$ are from the sort $S$, there is a formula $\theta(z)$ such that $\models\theta(e)$ if and only if $\phi(x,y,e)$ defines the graph of a bijection $S\to S$ which extends by identity to an automorphism of $\mathbb U'$.
Indeed, $\theta(z)$ asserts that $\phi(x,y,z)$ agrees with some element of $\Aut(\mathbb U'/\mathbb U)$ on $S$.

Now, by internality, for some $c\in\mathbb U'$, there is a $c$-definable bijection $f_c:S\to O_c$, where $O_c$ is a definable set in $\mathbb U$.
Let $\psi(x,u,v)$ be an $L'$-formula such that $\psi(x,u,c)$ defines the graph of $f_c$, and $\psi(x,u,c')$ defines the graph of a bijection $f_{c'}:S\to O_{c'}$ for all $c'$.
Let $z=(v,w)$ where $w$ is of the same arity as $v$, and set
$$\phi(x,y,z):=\exists u\big( \psi(x,u,v)\wedge\psi(y,u,w)\big).$$
Applying the property described in the paragraph above to $\phi(x,y,z)$, we have an $L'$-formula $\theta(v,w)$ such that $\models\theta(e,d)$ if and only if the bijection $f_d^{-1}\circ f_e:S\to S$ extends by identity to an automorphism of $\mathbb U'$.
(At first $\theta$ may be an $(L')^{\eq}$-formula, but as $v$ and $w$ belong to sorts in $\mathbb U'$ it can be replaced by an $L'$-formula.)
It follows by compactness that there must be a finite sublanguage $\Sigma\subseteq L'$ such that $f_d^{-1}\circ f_e\cup\id_{\mathbb U}$ is $L'$-elementary if and only if it is $\Sigma$-elementary.
Extending $\Sigma$ we may also assume that $\psi(x,u,v)$ is a $\Sigma$-formula.
Now set $L_\circ':=L\cup\Sigma$.

We claim that this $L_\circ'$ works.
That is, that every $0$-definable set in $\mathbb U'$ is $0$-definable in the reduct to $L'_\circ$.
Letting $\mathbb U'_\circ$ be the reduct of $\mathbb U'$ to the language $L'_\circ$, it suffices to show that $\Aut(\mathbb U'_\circ)=\Aut(\mathbb U')$, see for example~\cite[$\S$10.5]{hodges}.
Since $L'_\circ$ extends $L$, this in turn reduces to showing that $\Aut(\mathbb U'_\circ/\mathbb U)=\Aut(\mathbb U'/\mathbb U)$, see for example~\cite[Lemma~1.3]{hrushovski}.
Suppose $\sigma\in \Aut(\mathbb U'_\circ/\mathbb U)$.
Then for any $a\in S$,
\begin{eqnarray*}
f_c(a)
&=& \sigma(f_c(a)) \ \ \ \ \text{ as $f_c(a)\in\mathbb U$}\\
&=& f_{\sigma(c)}(\sigma(a)) \ \ \ \text{ as $\psi(x,u,v)$ is an $L'_{\circ}$-formula}
\end{eqnarray*}
so that $\sigma=f_{\sigma(c)}^{-1}\circ f_c\cup\id_{\mathbb U}$.
Hence, as $\sigma$ is $\Sigma$-elementary, it must be $L'$-elementary, as desired.
\end{proof}

\begin{corollary}
\label{wstable-finite-language}
If $\mathbb U'$ is an internal cover of $\mathbb U$ with one new sort $S$, and $\mathbb U'$ is totally transcendental, then the language of $\mathbb U'$ is finitely generated over $\mathbb U$.
\end{corollary}

\begin{proof}
The usual liason group theorem, see for example~\cite[Theorem 7.4.8]{pillay}, applied to totally transcendental theories, tells us that $\Aut(\mathbb U'/\mathbb U)$ together with its action on $S$ is interpretable in $\mathbb U'$.
Now apply Proposition~\ref{finite-language}.
\end{proof}

It follows that instances of internality in a totally transcendental theory can be understood in terms of internal generalised imaginary sorts.

\begin{example}
\label{dcfexample}
Suppose $\mathbb M$ is a saturated differentially closed field in characteristic zero and $\mathbb U$ is its field of constants with the induced structure (which we know is that of a pure algebraically closed field).
Suppose $S$ is a definable set in $\mathbb M$ that is internal to the constants, for example defined by $\delta x=1$ or $\delta x=x$.
Let $\mathbb U'$ be the $0$-definable structure induced by $\mathbb M$ on $(\mathbb U,S)$.
Then, by Corollary~\ref{wstable-finite-language}, $\mathbb U'$ is an internal generalised imaginary sort of $\mathbb U$.
One gets similar examples by working with the theory of compact complex manifolds rather than differentially closed fields, in which case $\mathbb U$ is the induced structure on the projective line (again essentially a pure algebraically closed field).
\end{example}

Having proved the existence of liason groupoids, Hrushovski goes on to show in~\cite[Theorem~3.2]{hrushovski} that every $0$-definable connected groupoid in $\mathbb U$ comes from an internal generalised imaginary sort.
This is not precisely correct.
Note that the groupoid $\mathcal G'$ of Theorem~\ref{groupoid-construction} has an additional property: if $c \in \mathbb U'$ is as in the proof, then any automorphism $\sigma \in
\Aut(\mathbb U'/\mathbb U)$ that fixes $c$ must be the identity.
Indeed, for any $a\in S$,
\begin{eqnarray*}
\sigma(a)
&=&
\sigma(f_c^{-1}f_c(a))\\
&=&
f_{\sigma(c)}^{-1}(\sigma(f_c(a)))\\
&=&
f_c^{-1}(f_c(a))\\
&=&
a.
\end{eqnarray*}
By connectedness, a similar property then holds for $\mathcal G$, the restriction of $\mathcal G'$ to $\mathbb U$.
So the definable groupoids in $\mathbb U$ that arise this way are special.

\begin{definition}
A definable groupoid $\mathcal G$ is called {\em finitely faithful}\footnote{This terminology was suggested to us by Hrushovski in private communication.} if for any object $O$ of $\mathcal G$, there is a finite $c \subset O$ such that the pointwise stabiliser of $c$ in $\Aut_{\mathcal G}(O)$ is trivial.
\end{definition}

Note that if $\mathcal G$ is connected then it suffices to check the condition for some (rather than every) object~$O$.

Of course, not all definable groupoids (even group actions) are finitely faithful.

From a $0$-definable connected groupoid $\mathcal G$ in $\mathbb U$, not necessarily finitely faithful, Hrushovski does construct an associated internal cover of $\mathbb U$, but it is not an internal generalised imaginary sort as it involves two new sorts (rather than one), neither of which can be eliminated in general.
We review his construction and then explain how finite faithfulness enters the picture.

Let $\mathcal G = (O_i, f_m)_{i \in I, m \in M}$ be a $0$-definable connected
groupoid in $\mathbb U$.
We want to associate an internal cover $\mathbb U'$.
The idea is to construct $\mathbb U'$ by adding a single new object to $\mathcal G$.
We just duplicate one of the objects and add the necessary morphisms to make this into a connected groupoid.
Fix some $i \in I$.\label{constructiondetails}
There are two new sorts $O_*$ and $M_*$. The sort
$O_*$ as a set is a copy of $O_i$, for concreteness let $O_* = O_i \times \{0\}$.
The sort $M_*$ is a copy of the set of morphisms in $M$ whose domain or range is
$O_i$, plus two additional copies of $M(i,i)$, for concreteness say
$$M_* = (M(i,i) \times \{0, 1, 2\}) \cup \bigcup_{j \neq i \in J} \big(M(i,j) \cup
M(j,i)\big) \times \{0\}.$$
The idea is that
\begin{itemize}
\item
$M(i,i)\times\{0\}$ will encode morphisms $O_*\to O_*$,
\item
$M(i,i)\times\{1\}$ will encode morphisms $O_i\to O_*$,
\item
$M(i,i)\times\{2\}$ will encode morphisms $O_*\to O_i$,
\item
$M(i,j)\times\{0\}$ will encode morphisms $O_*\to O_j$, and
\item
$M(j,i)\times\{0\}$ will encode morphisms $O_j\to O_*$.
\end{itemize}
Next we add a ternary relation $R$ that interprets
these morphisms as bijection.
Formally $R$ is interpreted as the set $R_1\cup R_2\cup R_3$ where
\begin{itemize}
\item
$R_1\subset M_*\times O\times O_*$ is given by
\begin{eqnarray*}
R_1&:=&\big\{\big((m,1),a,(b,0)\big):m\in M(i,i), a,b\in O_i, f_m(a)=b\big\}\ \cup\\
& & \big\{\big((m,0),a,(b,0)\big):m\in M(j,i), j\neq i, a\in O_j, b\in O_i, f_m(a)=b\big\}
\end{eqnarray*}
\item
$R_2\subset M_*\times O_*\times O$ is given by
\begin{eqnarray*}
R_1&:=&\big\{\big((m,2),(a,0),b\big):m\in M(i,i), a,b\in O_i, f_m(a)=b\big\}\ \cup\\
& & \big\{\big((m,0),(a,0),b\big):m\in M(i,j), j\neq i, a\in O_i, b\in O_j, f_m(a)=b\big\}
\end{eqnarray*}
\item
$R_3\subset M_*\times O_*\times O_*$ is given by
\begin{eqnarray*}
R_1&:=& \big\{\big((m,0),(a,0),(b,0)\big):m\in M(i,i), a,b\in O_i, b\in O_j, f_m(a)=b\big\}
\end{eqnarray*}
\end{itemize}

\begin{remark}
There is a notational detail here: the arity of $R$ is not well defined as the sequence of sorts are different in $R_1$, $R_2$ and $R_3$.
So technically we either view this as three seperate relations, or what we prefer, is to work in $(\mathbb U')^{\eq}$ where  $O \sqcup O_*$ can be identified with a sort, and then view  $R \subseteq M_*\times (O
\sqcup O_*) \times (O \sqcup O_*)$.
\end{remark}

We define $\mathcal G'$ to be the groupoid with objects $(O_i)_{i\in I}\cup\{O_*\}$ and morphisms $(f_m)_{m\in M}\cup(R_{m'})_{m'\in M_*}$.
It is a connected one-object extension of $\mathcal G$.
Working in $(\mathbb U')^{\eq}$, we set $I' := I \cup \{*\}$, $O':=O \sqcup O_*$, $M':=M\sqcup M_*$, $\pi':O'\to I'$ the natural extension of $\pi:O\to I$ by $O_*$, and $f'\subseteq O'\times O'\times M'$ the natural extension of $f\subseteq O\times O\times M$ by $R$.
It is easily seen that this presents $\mathcal G'$ as a $0$-definable connected groupoid in $(\mathbb U')^{\eq}$.

\begin{theorem} [Hrushovski]
\label{cover-construction}
Given $\mathbb U$ and $0$-definable connected groupoid $\mathcal G$, the structure $\mathbb U'$ constructed above determines an internal cover of $\mathbb U$ that depends only on $\mathcal G$ (so not on the choice of $i$).
Furthermore, $\Aut_{\mathcal G'}(O_*)$ together with its action on $O_*$ is isomorphic to 
$\Aut(O_*/\mathbb U):=\{\sigma|_{O_*}:\sigma\in\Aut(\mathbb U')\text{ and }\sigma|_{\mathbb U}=\id_{\mathbb U}\}$.
\end{theorem}

We leave the reader to consult~\cite[Lemma~3.1]{hrushovski} for a proof.

We wish to eliminate the sort $M_*$ so as to get an internal generalised imaginary sort.
Hrushovski erroneously claims that this is always possible (see the parenthetical final sentence of the first paragraph of the proof of Theorem~3.2 of~\cite{hrushovski}).
The next result explains when in fact this can be done.

\begin{proposition}
\label{whenff}
Let $\mathbb U'$ be as constructed above.
The following are equivalent.
\begin{enumerate}
\item The groupoid $\mathcal G$ is finitely faithful.
\item $M_* \subseteq \dcl(\mathbb U, O_*)$ where $\dcl$ is calculated in~$\mathbb
U'$.
\item $(\mathbb U, O_*)$
is stably embedded in~$\mathbb U'$.
\end{enumerate}
\end{proposition}

\begin{proof}
If $\mathcal G$ is finitely faithful then, as $\mathcal G'$ is a connected one-object extension of~$\mathcal G$, we get that $\mathcal G'$ is finitely faithful also.
For each $i\in I'$, let $c_i\subset O'_i$ be a finite tuple whose pointwise stabiliser in $\Aut_{\mathcal G'}(O'_i)$ is trivial.
Then for any $m\in M'$, say $m\in M'_{i,j}$ for some $i,j\in I'$, we have that $f_m':O'_i\to O'_j$ is determined by its action on $c_i$, and so $m\in \dcl(f'_m(c_i), c_i)$.
Hence $(1)\implies(2)$.

That $(2) \implies (3)$ is clear.

Let us show that $(3) \implies (1)$.
By connectedness, to show that $\mathcal G$ is finitely faithful it suffices to show that $\mathcal G'$ is, and for this we need only check that $O_*$ has a finite subtuple whose pointwise stabiliser is trivial.

Fix $i\in I$.
Note that for any $m,n \in M'_{*,i}$, 
$(f_n')^{-1}f'_m\in\Aut(O_*/\mathbb U)$ by Theorem~\ref{cover-construction}, and so extends to some $\sigma\in\Aut(\mathbb U'/\mathbb U)$.
Moreover, for any $b\in O_*$
\begin{eqnarray*}
f'_{\sigma(m)}(b)
&=&
\sigma(f'_m(\sigma^{-1}(b)))\\
&=&
f'_m(\sigma^{-1}(b))\ \ \text{ as $f'_m(\sigma^{-1}(b))\in O_i\subseteq\mathbb U$}\\
&=&
f'_m((f'_m)^{-1}f'_n(b)))\\
&=&
f'_n(b).
\end{eqnarray*}
So $\sigma(m)=n$.

Now, fix $m\in M'_{*,i}$.
By stable embeddedness there is a small $A \subseteq O_*$
such that $\tp(m/\mathbb U A) \vdash \tp(m/\mathbb U O_*)$ (see \cite{zoe-udi}).
If $f'_m$ and $f'_n$ agree on $A$ then the $\sigma$ constructed above fixes $A$ pointwise,  and so $m$ and $n$ have the same type over $\mathbb U\cup O_*$, which means that $f'_m$ and $f'_n$ agree on all of $O_*$, and hence $m=n$.
It follows that $m$ is fixed by all automorphism in
$\Aut(\mathbb U'/A f'_m(A))$.
By saturation, $m$ is fixed by all automorphism in
$\Aut(\mathbb U'/a f'_m(a))$ for some finite tuple $a$ from  $A$.
If $f'_m$ and $f'_n$ agree on $a$ then $\sigma\in \Aut(\mathbb U'/a f'_m(a))$, and so we must have $n=m$.
But this means that the pointwise stabiliser of $a$ in $\Aut_{\mathcal
G'}(O')$ is the identity, else one could precompose by a nontrivial element of this stabiliser.
So $\mathcal G'$ is finitely faithful.
\end{proof}

\begin{corollary}
If $\mathbb U$ is stable then all definable connected groupoids are finitely faithful.
\end{corollary}

\begin{proof}
Note that an internal cover of a stable theory is stable as any formula with the order property in the cover can be pushed into the reduct using any definable bijection witnessing the internality.
Now if $\mathbb U'$ is stable then condition~(3) of Proposition~\ref{whenff} holds automatically.
\end{proof}

In any case, we now can refine Theorem~\ref{cover-construction} to

\begin{theorem}
\label{cover-construction-refined}
Suppose $\mathcal G$ is a $0$-definable finitely faithful connected groupoid in~$\mathbb U$, and let $\mathbb U'$ be the internal cover constructed in Theorem~\ref{cover-construction}.
Let $\mathbb U''$ be the structure induced on $(\mathbb U,O_*)$ by $\mathbb U'$.
Then $\mathbb U''$ is an internal generalised imaginary sort of $\mathbb U$, $\mathcal G'$ is $0$-definable in $(\mathbb U'')^{\eq}$, and $\Aut_{\mathcal G'}(O_*)$ with its action on $O_*$ is isomorphic to  $\Aut(\mathbb U''/\mathbb U)$.
\end{theorem}

\begin{proof}
Clearly $\mathbb U''$ is an internal cover involving a single new sort $O_*$.
Note also that by condition~(2) of Proposition~\ref{whenff}, and our assumption that $\mathcal G$ is finitely faithful, $\mathbb U'$ is interpretable in $\mathbb U''$.
Hence $\mathcal G'$ is interpretable in $\mathbb U''$, and $\Aut(\mathbb U''/\mathbb U)=\Aut(O_*/\mathbb U)$ where the latter is in the sense of $\mathbb U'$.
So the liason group statement follows from Theorem~\ref{cover-construction}.
The only thing left to check is that the language of $\mathbb U''$ is finitely generated over that of $\mathbb U$.
But this follows from Proposition \ref{finite-language} since the liason group and its action is interpretable.
\end{proof}

This yields a bijective correspondence which appears as Theorem~3.2 in~\cite{hrushovski}, though with the assumption of finite faithfulness omitted.

\begin{theorem}[Hrushovski]
\label{correspondence}
The constructions of Theorems~\ref{groupoid-construction} and~\ref{cover-construction-refined} form a bijective correspondence between internal generalised imaginary sorts of $\mathbb U$ and $0$-definable finitely faithful connected groupoids in $\mathbb U$, up to equivalence on both sides.
\end{theorem}

We leave the reader to consult~\cite{hrushovski} for a proof, but we should at least spell out what these equivalence are. Two generalised imaginary sorts
$\mathbb U_1=(\mathbb U, S_1,\dots)$ and $\mathbb U_2=(\mathbb U, S_2,\dots)$ are called {\em equivalent} if there is a bijection $g : S_1 \to S_2$ such that $g \cup
\id_{\mathbb U}$ is a bi-interpretation; that is, $X$ is $0$-definable in $\mathbb U_1$ if and only if its image is $0$-definable in $\mathbb U_2$.

For equivalence of definable groupoids, recall from~\cite{hrushovski} that a definable embedding of $\mathcal
G_1$ into $\mathcal G_2$ consists of the following:
\begin{itemize}
\item a definable injection $\iota : I_1 \to I_2$;
\item a definable family of definable bijections $(h_i : O_{1i} \to O_{2\iota(i)})_{i \in I_1}$
such that for every $i,j \in I_1$ we have $h_j \circ \Hom_{\mathcal
G_1}(O_{1i}, O_{1j}) = \Hom_{\mathcal G_2}(O_{2\iota(i)}, O_{2\iota(j)}) \circ h_i$.
\end{itemize}
Now, two $0$-definable connected groupoids $\mathcal G_1$ and $\mathcal G_2$ are called {\em equivalent} if there is a $0$-definable
connected groupoid $\mathcal G$ and $0$-definable embeddings $\mathcal G_i \to
\mathcal G$ for $i=1,2$.

\bigskip
\section {Functoriality}
\label{section-functorial}

\noindent
In this section we categorify the correspondence of~\ref{correspondence}.
Our goal here is twofold.  Firstly, one rarely studies
internal covers in isolation. For example in the monster model of $\mathrm
{DCF}_0$ one looks at definable sets internal to the constants, but also at
definable maps between them. So a question arises as to what kind of object do
these maps induce between the corresponding groupoids in the constants.
Secondly, the
correspondence of~\ref{correspondence} is defined up to equivalence; 
we would like to realise these equivalences as isomorphisms.

For a general reference on terminology from category theory we suggest~\cite{leinster}.

\medskip
\subsection{Morphisms of covers}
\label{subsectcovmor}
The idea is that morphisms should be definable maps in a common cover.
So a {\em morphism} between covers $\mathbb U_1=(\mathbb U, S_1,\dots)$ and $\mathbb U_2=(\mathbb U, S_2,\dots)$, each with a single new sort, is given by a common cover of $\mathbb U_1$ and $\mathbb U_2$ of the form $(\mathbb U_1\cup\mathbb U_2,g)$ where $g:S_1\to S_2$ is a new basic function.
(See~$\S$\ref{coversection} for a discussion of unions of covers.)

It may be worth being careful about the language here.
Suppose $L$ is the language of $\mathbb U$ and $L_i$ is the langauge of $\mathbb U_i$ for $i=1,2$.
We form a copy of $L_1$, which we denote by $L_1^{\dom}$, by replacing each symbol $R\in L_1\setminus L$ by the symbol $R^{\dom}$.
Similarly we make a duplicate of $L_2$ which we denote by $L_2^{\ran}$.
We then have that $L_1^{\dom}\cap L_2^{\ran}=L$.
The morphism $(\mathbb U_1\cup\mathbb U_2,g)$ is a structure in the language $L_1^{\dom}\cup L_2^{\ran}\cup\{g\}$, where $g$ is a new function symbol from $S_1^{\dom}$ to $S_2^{\ran}$.
So when we say that $(\mathbb U_1\cup\mathbb U_2,g)$ is a cover of $\mathbb U_1$ and $\mathbb U_2$, we really mean modulo the natural identification of $L_1$ with $L_1^{\dom}$ and $L_2$ with $L_2^{\ran}$.

\begin{remark}
\label{functions-covers}
Of course the morphism is really the theory of $(\mathbb U_1\cup\mathbb U_2,g)$, so that different functions may produce the same morphism. 
But note that $(\mathbb U_1\cup\mathbb U_2, g) \equiv (\mathbb U_1\cup\mathbb U_2, h)$ if and only if $h=\alpha g$ for some $\alpha\in\Aut(\mathbb U_2/\mathbb U)$.
Indeed, by saturation there is an ismorphism $\tau:(\mathbb U_1\cup\mathbb U_2, g) \to (\mathbb U_1\cup\mathbb U_2, h)$, and by stable embeddedness of $\mathbb U_1$ in $(\mathbb U_1\cup\mathbb U_2, g)$ we can take $\tau$ to be the identity on~$\mathbb U_1$.
Then $\alpha:=\tau|_{\mathbb U_2}$ works.
\end{remark}

Composition is defined naturally: given morphisms $(\mathbb U_1\cup\mathbb U_2,g)$ and $(\mathbb U_2\cup\mathbb U_3,h)$, the composition is $(\mathbb U_1\cup\mathbb U_3, h\circ g)$.
To see that this is a well-defined morphism we need to check that it is a cover of $\mathbb U_1$ and $\mathbb U_3$ and that its theory does not depend on the particular choice of $g$ and $h$ (just on the morphisms that $g$ and $h$ determine).
We do this by applying Lemma~\ref{union} to $(\mathbb U_1\cup\mathbb U_2,g)$ and $(\mathbb U_2\cup\mathbb U_3,h)$, which are both covers of~$\mathbb U_2$, at least after re-identifying $L_2$, $L_2^{\ran}$, and $L_2^{\dom}$, as say $L_2^{\operatorname{mid}}$.
We get that $(\mathbb U_1\cup\mathbb U_2\cup\mathbb U_3, g, h)$ has a theory that depends only on the theories of $(\mathbb U_1\cup\mathbb U_2,g)$ and $(\mathbb U_2\cup\mathbb U_3,h)$, and is a cover of both $\mathbb U_1$ and~$\mathbb U_3$.
Now, $(\mathbb U_1\cup\mathbb U_3, h \circ g)$ is a reduct of $(\mathbb U_1\cup\mathbb U_2\cup\mathbb U_3, g, h)$ expanding $\mathbb U_1$ and~$\mathbb U_3$, and so it too has these desired properties.

The identity morphism on a cover $\mathbb U'=(\mathbb U,S,\dots)$ is $(\mathbb U'\cup\mathbb U',\id_S)$, where $\id_S$ is here viewed as a function from $S^{\dom}$ to $S^{\ran}$, which are the two new sort symbols in the language of $\mathbb U'\cup\mathbb U'$, both of which are interpreted in $\mathbb U'\cup\mathbb U'$ as the set $S$.
To see that $(\mathbb U'\cup\mathbb U',\id_S)$ is saturated, let $(\mathbb U_1\cup\mathbb U_2,g)$ be a saturated model of its theory.
Note that while it is no longer the case that on underlying sets $g$ is the identity -- indeed the underlying sets of the new sorts of $\mathbb U_1$ and $\mathbb U_2$ need not be identical -- it is the case that modulo the natural identification of $L'^{\dom}$ with $L'^{\ran}$, $g\cup\id_{\mathbb U}$ is an isomorphism from $\mathbb U_1$ to $\mathbb U_2$.
Also, $\mathbb U_i$ is a cover of $\mathbb U$ of the same theory as $\mathbb U'$, so that by Lemma~\ref{union}, we have an isomorphism $\sigma_1\cup\sigma_2:\mathbb U'\cup\mathbb U'\to \mathbb U_1\cup\mathbb U_2$.
In order to be an isomorphism between $(\mathbb U'\cup\mathbb U',\id_S)$ and $(\mathbb U_1\cup\mathbb U_2, g)$ we would need the following diagram to commute:
\begin{center}
\begin{tikzcd}
S\ar[rr, "\id_S"]\ar[d, "\sigma_1|_{S}"]&&S\ar[d, "\sigma_2|_{S}"]\\
S_1\ar[rr,"g"]&&S_2
\end{tikzcd}
\end{center}

While this may not be the case, $\alpha:=(\sigma_2|_{S})^{-1}\circ g\circ\sigma_1|_{S}$
is such that $\alpha\cup\id_{\mathbb U}$ is an $L'^{\ran}$-automorphism of $\mathbb U'$.
So replacing $\sigma_2$ by $\sigma_2\circ(\alpha\cup\id_{\mathbb U})$, we get that the diagram does commute.
Hence $(\mathbb U'\cup\mathbb U',\id_S)$ is isomorphic to $(\mathbb U_1\cup\mathbb U_2,g)$, and is therefore saturated.
Similar arguments show that $\mathbb U'$, under both the identification of $L'$ with $L'^{\dom}$ and with $L'^{\ran}$, is stably embedded in $(\mathbb U'\cup\mathbb U',\id_S)$.

This defines the category of covers of $\mathbb U$ with one extra
sort.
The following identifies the isomorphisms of this category.

\begin{proposition}
\label{isomorphism}
A morphism $(\mathbb U_1\cup\mathbb U_2,g)$ in this category of covers of $\mathbb U$ by a single sort is an isomorphism if and only if $g\cup\id_{\mathbb U}:\mathbb U_1\to\mathbb U_2$ is a bijective bi-interpretation.
\end{proposition}

\begin{proof}
It is not hard to verify that in this category a morphism $(\mathbb U_1\cup\mathbb U_2,g)$ is an isomorphism if and only if $g:S_1\to S_2$ is bijective.

Suppose $(\mathbb U_1\cup\mathbb U_2,g)$ is an isomorphism and $X \subseteq S_1^n \times \mathbb U^m$ is $0$-definable.
Then
its image $B = (g \cup \id_{\mathbb U})(A) \subseteq S_2^n \times \mathbb U^m$
is $0$-definable in $(\mathbb U_1\cup\mathbb U_2, g)$.
Hence by stable embeddedness and saturation it is $0$-definable in $\mathbb U_2$.
Similarly, preimages of $0$-definable sets in $\mathbb U_2$ are $0$-definable in
$\mathbb U_1$.
So $g\cup\id_{\mathbb U}$ is a bijective bi-interpretation.

Conversely, suppose $g\cup\id_{\mathbb U}$ be a bijective bi-interpretation.
We need to show that $g$ is
a morphism, i.e. $(\mathbb U_1\cup\mathbb U_2, g)$ is a cover of both $\mathbb U_1$ and $\mathbb U_2$.

For saturation, we let $(\overline{\mathbb U}_1\cup\overline{\mathbb U}_2, \overline g)$ be a saturated model of its theory.
Then Lemma~\ref{union} gives us an $(L_1^{\dom}\cup L_2^{\ran})$-isomorphism $\sigma_1\cup\sigma_2:\mathbb U_1\cup\mathbb U_2\to \overline{\mathbb U}_1\cup\overline{\mathbb U}_2$.
Again, to be an isomorphism from $(\mathbb U_1\cup\mathbb U_2, g)$ to $(\overline{\mathbb U}_1\cup\overline{\mathbb U}_2, \overline g)$ we would need
\begin{center}
\begin{tikzcd}
S_1\ar[rr,"g"]\ar[d, "\sigma_1|_{S_1}"]&&S_2\ar[d, "\sigma_2|_{S_2}"]\\
\overline{S}_1\ar[rr, "\overline g"]&&\overline{S}_2
\end{tikzcd}
\end{center}
to commute.
We can make the diagram commute by replacing $\sigma_2$ by $\sigma_2\circ(\alpha\cup\id_{\mathbb U})$ where $\alpha:=(\sigma_2|_{S_2})^{-1}\circ\overline g\circ\sigma_1|_{S_1}\circ g^{-1}$.
But note that after such an alteration $\sigma_1\cup \sigma_2$ is no longer elementary with respect to $L_2^{\ran}$.
However, $\sigma_1\cup\sigma_2:(\mathbb U_1\cup\mathbb U_2, g)\to (\overline{\mathbb U}_1\cup\overline{\mathbb U}_2,\overline g)$ still has the property that a set is $0$-definable if and only if its image is.
This is enough to conclude that $(\mathbb U_1\cup\mathbb U_2, g)$ is saturated since $(\overline{\mathbb U}_1\cup\overline{\mathbb U}_2, \overline g)$ is.

Next we show that $\mathbb U_1$ is stably embedded in $(\mathbb U_1\cup\mathbb U_2, g)$, modulo the identification of $L_1$ with $L_1^{\dom}$.
Given $\sigma_1\in\Aut(\mathbb U_1)$ we can extend it to $\sigma_1\cup\sigma_2$ on $\mathbb U_1\cup\mathbb U_2$ by Lemma~\ref{union}.
To make it commute with $g$ we would precompose $\sigma_2$ with by $(\sigma_2|_{S_2})^{-1}\circ g\circ\sigma_1|_{S_1}\circ g^{-1}$.
This time the new $\sigma_1\cup\sigma_2$ is elementary, as we are conjugating by the bi-interpretation $g$.
Similarly $(\mathbb U, S_2)$ is stably embedded in $(\mathbb U_1\cup\mathbb U_2, g)$, modulo the identification of $L_2$ with $L_2^{\ran}$.
\end{proof}

We are interested not in arbitrary covers, but internal generalised imaginaries.

\begin{definition}
We denote by $\mathbf {IGI}_T$ the full subcategory generated by the internal generalised imaginary sorts.
\end{definition}

Proposition~\ref{isomorphism} says that isomorphisms in $\mathbf {IGI}_T$ are precisely equivalences of internal generalised imaginary sorts as defined by Hrushovski~\cite{hrushovski} (and discussed in the previous section).

\medskip
\subsection{Morphisms of definable groupoids}
\label{subsectgpdmor}
If we think of a definable groupoid as a category we might be tempted to have the morphisms be functors, maybe satisfying additional properties, but in any case involving a definable function between the sets of objects.
If we focus on the definable groupoid as generalising definable group actions we may expect morphisms to induce definable group homomorphisms between the automorphism groups.
As it turns out, both of these approaches are misleading.
If, as is the case, we view definable groupoids as the intrinsic avatars of internal generalised imaginaries, then we should proceed as follows.

Fix connected groupoids $\mathcal G_1 = (O_{1i}, f_{1m})_{i \in I_1, m \in M_1}$ and $\mathcal G_2 = (O_{2i}, f_{2m})_{i\in I_2, m \in M_2}$ definable in $\mathbb U$.
A {\em definable morphism} $h:\mathcal G_1\to\mathcal G_2$ is a set of uniformly definable functions $\big(h_n : O_{1i_1} \to O_{2i_2}\big)_{n\in N}$ satisfying conditions~(A) and~(B) below.
More precisely, we have a definable set $N$ equipped with a definable map $N\to I_1\times I_2$, and a definable set $\displaystyle h\subseteq N\times O_1\times O_2$, such that for every $n\in N$ living above $(i_1,i_2)\in I_1\times I_2$, the fibre $h_n\subseteq O_{1i_1}\times O_{2i_2}$ is the graph of a function.
We denote by $N(i_1,i_2)$ the fibre of $N\to I_1\times I_2$ above $(i_1,i_2)$, so that $h|_{N(i_1,i_2)}$ is a definable family of definable functions from $O_{1i_1}$ to $O_{2i_2}$.
The conditions we impose on the family $h$ to be a morphism are:
\begin{itemize}
\item[(A)]
For all $p\in N(i_1,i_2), q\in N(j_1,j_2)$,
$$h_{q} \circ \Hom_{\mathcal G_1}(O_{1i_1}, O_{1j_1}) = \Hom_{\mathcal G_2}(O_{2i_2}, O_{2j_2})\circ h_{p}.$$
\item[(B)]
For all $p\in N(i,j)$ and $m\in M_2(j,k)$, there is $q\in N(i,k)$ such that
$$f_{2m}h_p=h_q.$$
\end{itemize}
Note that condition~(B) says $h$ is closed under compositions with bijections from $\Hom_{\mathcal G_2}$.
Given~(A) this implies $h$ is also closed under precompositions with $\Hom_{\mathcal G_1}$.

\begin{remark}
\begin{itemize}
\item[(a)]
To be precise, for us the definable morphism is the set of definable functions and not the definable family.
That is we identify $h$ and~$h'$ if
$\{h_n : n \in N\} = \{h'_{n} : n\in N'\}$.
\item[(b)]
We do not require each $n \in N$ to define a distinct morphism, though this could be easily achieved by elimination of imaginaries.
\end{itemize}
\end{remark}

We now compose morphisms.
Suppose $g : \mathcal G_1 \to \mathcal G_2$ and $h : \mathcal G_2 \to \mathcal G_3$ are definable morphisms between definable connected groupoids, where $g=(g_p)_{p\in P}$ and $h=(h_q)_{q\in Q}$.
Then $h\circ g:\mathcal G_1\to\mathcal G_3$ is given as follows.
Let $N:=P\times_{I_2}Q$ be the fibred product of $P\to I_1\times I_2\to I_2$ and $Q\to I_2\times I_3\to I_2$, and equip it with the natural map $N\to I_1\times I_3$.
Given $n\in N(i_1,i_3)$ we have that for some $i_2\in I_2$, $n=(p,q)$ where $p\in P(i_1,i_2)$ and $q\in Q(i_2,i_3)$.
Then
$$(h\circ g)_n:=h_q\circ g_p:O_{1i_1}\to O_{3i_3}.$$

\begin{lemma}
The composition $h\circ g$ is a definable morphism from $\mathcal G_1$
to $\mathcal G_3$
\end{lemma}

\begin{proof}
It is more or less clear that $h\circ g$ is a set of uniformly definable functions parameterised by $N\to I_1\times I_3$ with the correct domains and codomains.
We need to check conditions~(A) and~(B).

For~(A) suppose we are given $n\in N(i_1,i_3)$ and $n'\in N(j_1,j_3)$.
Then $n=(p,q)$ and $n'=(p',q')$ where $p\in P(i_1,i_2)$, $q\in Q(i_2,i_3)$, $p'\in P(j_1,j_2)$ and $q'\in Q(j_2,j_3)$, for some $i_2, j_2\in I_2$.
Now,
\begin{eqnarray*}
(hg)_{n'}\Hom_{\mathcal G_1}(O_{1i_1}, O_{1j_1})
&=&
h_{q'}g_{p'} \Hom_{\mathcal G_1}(O_{1i_1}, O_{1j_1})\\
&=&
h_{q'} \Hom_{\mathcal G_2}(O_{2i_2}, O_{2j_2})g_p \ \ \text{by~(A) for $g$}\\
&=&
\Hom_{\mathcal G_3}(O_{3i_3}, O_{3j_3})h_qg_p \ \ \text{by~(A) for $h$}\\
&=&
\Hom_{\mathcal G_3}(O_{3i_3}, O_{3j_3})(hg)_n
\end{eqnarray*}
as desired.

For~(B), given $(hg)_n$ and $f_{3m}$ we seek $(hg)_{n'}$ such that the following commutes:
\begin{center}
\begin{tikzcd}
O_{1i}\ar[dd, "(hg)_n"'] \ar[ddr, dashed, "(hg)_{n'}"] \\
\\
O_{3j}\ar[r,"f_{3m}"] & O_{3k}
\end{tikzcd}
\end{center}
Write $n=(p,q)$ where $p\in P(i,\ell)$ and $q\in Q(\ell,j)$ for some $\ell\in I_2$.
Then the above becomes
\begin{center}
\begin{tikzcd}
O_{1i}\ar[d, "g_p"']  \\
O_{2\ell}\ar[d,"h_q"']\ar[dr, "h_{q'}"]\\
O_{3j}\ar[r,"f_{3m}"] & O_{3k}
\end{tikzcd}
\end{center}
where the existence of $h_{q'}$ is given by~(B) applied to $h$.
So $n':=(p,q')$ will work.
\end{proof}

The identity morphism for a definable connected groupoid $\mathcal G = (O_i, f_m)_{i \in I, m \in M}$ is the
family $f=(f_m)_{m \in M}$ itself.
That $f$ is a definable endomorphism, namely that~(A) and~(B) hold, follows easily from the fact that the $f_m$ are all bijections, and that $f$ is closed under compositions and inverses.
That it really serves as the identity is also straightforward: 
given a definable morphism $h:\mathcal H\to\mathcal G$, that $f\circ h \subseteq h$ is exactly~(B) for $h$, and that $f\circ h\supseteq h$ is by the fact that we have all identities in $f$.
A similar argument shows that $h\circ f=h$ for any definable morphism $h:\mathcal G\to\mathcal H$, using the fact observed before that~(B) holds with precompositions too.

We thus have a category;
let $\mathbf {CGpd}_T$ denote the $0$-definable connected groupoids in $T$ equipped with the $0$-definable morphisms.

\begin{lemma}
\label{iso=bij}
A $0$-definable morphism $h : \mathcal G_1 \to \mathcal G_2$ is an isomorphism in $\mathbf {CGpd}_T$ if and
only if each (equivalently some) $h_n$ is a bijection.
\end{lemma}

\begin{proof}
If some $h_n$ is bijective then so are all the others, since by~(A) they can be obtained from $h_n$ by
pre and post-composing with bijections in $\Hom_{\mathcal G_1}$ or $\Hom_{\mathcal G_2}$.

If every $h_n$ is bijective then it is easy to see that $h^{-1}:=(h^{-1}_n)_{n\in N}$ is a definable morphism parameterised by $N\to I_2\times I_1$, where this is the co-ordinate flip of the map $N\to I_1\times I_2$ that parameterised $h$.
That it is the inverse of $h$ also follows readily from~(A).
For example, if $p\in N(i,j)$ and $q\in N(k,j)$ then from 
$$h_{q} \circ \Hom_{\mathcal G_1}(O_{1i}, O_{1k}) = \Hom_{\mathcal G_2}(O_{2j}, O_{2j})\circ h_{p}$$
we get that $h_ph^{-1}_q\in \Hom_{\mathcal G_1}(O_{1i}, O_{1k})$ and hence is one of the definable functions in the identity endomorphism on $\mathcal G_1$.

Conversely, suppose $g=(g_p)_{p\in P}:\mathcal G_2\to\mathcal G_1$ is the inverse of $h$.
Then for any $n\in N(i,j)$ and $p\in P(j,i)$ we have that $g_ph_n\in \Aut_{\mathcal G_1}(O_{1i})$ is bijective, and so $h_n$ is injective.
But also $h_ng_p\in\Aut_{\mathcal G_2}(O_{2j})$ is bijective, implying $h_n$ is
surjective.
\end{proof}

We want to show that equivalences, as defined by Hrushovski~\cite{hrushovski} and discussed at the end of~$\S$\ref{section-covers-gpd} above, agree with isomorphisms in $\mathbf {CGpd}_T$.
Recall that two definable connected groupoids were said to be equivalent if they jointly embedded into a third definably connected groupoid.
So the key is to show that embeddings in Hrushovski's sense induce isomorphisms in our's.

Recall that a definable embedding from $\mathcal G_1$ to $\mathcal G_2$ is given by a definable injection $\iota : I_1 \to I_2$ and a definable family of definable bijections $(h_i : O_{1i} \to O_{2\iota(i)})_{i \in I_1}$
such that $h_j \circ \Hom_{\mathcal
G_1}(O_{1i}, O_{1j}) = \Hom_{\mathcal G_2}(O_{2\iota(i)}, O_{2\iota(j)}) \circ h_i$.
We can produce from this an isomorphism $\widetilde h=(\widetilde h_n)_{n\in N}:\mathcal G_1\to\mathcal G_2$ such that 
$$\widetilde h|_{N(i_1,i_2)}=\Hom_{\mathcal G_2}(O_{2\iota(i_1)},O_{2i_2})\circ h_{i_1}$$
To do so, consider $M_2\to I_2$ which picks out the domain of a bijection in $
\mathcal G_2$, and let $N$ be the fibred product of $M_2\to I_2$ with $\iota:I_1\to I_2$.
We have the natural map $N\to I_1\times I_2$.
Given $n\in N(i_1,i_2)$ write $n=(m,i_1)$ where $m\in M_2(\iota(i_1),i_2)$, and define $\widetilde h_n:=f_{2m}\circ h_{i_1}:O_{1i_1}\to O_{2i_2}$.

\begin{proposition}
Suppose $\mathcal G_1$ and $\mathcal G_2$ are definable connected groupoids and $(\iota, h_i)_{i\in I_1}$ is a definable embedding of $\mathcal G_1$ into $\mathcal G_2$.
The family of definable maps $\widetilde h$ given above is an isomorphism from $\mathcal G_1$ to $\mathcal G_2$ in $\mathbf {CGpd}_T$.
\end{proposition}

\begin{proof}
As each $\widetilde h_n$ is bijective by definition, Lemma~\ref{iso=bij} tells us that it suffices to show that $\widetilde h$ is a morphism, namely that it satisfies~(A) and~(B).

For~(A), fix $p\in N(i_1,i_2)$ and $q\in N(j_1,j_2)$.
Then $p=(m,i_1)$ for some $m\in M_2(\iota(i_1),i_2)$, and $q=(m',j_1)$ for some $m'\in M_2(\iota(j_1),j_2)$, and we have
\begin{eqnarray*}
\widetilde h_{q}\Hom_{\mathcal G_1}(O_{1i_1}, O_{1j_1})
&=&
f_{2m'}h_{j_1} \Hom_{\mathcal G_1}(O_{1i_1}, O_{1j_1}) \ \ \text{ by definition of }\widetilde h\\
&=&
f_{2m'} \Hom_{\mathcal G_2}(O_{2\iota(i_1)}, O_{2\iota(j_1)})h_{i_1} \ \ \text{ as $h$ is an embedding}\\
&=&
\Hom_{\mathcal G_2}(O_{2\iota(i_1)}, O_{2\iota(j_1)})f_{2m}h_{i_1} \ \ \text{ as $f_2$ satisfies~(A)}\\
&=&
\Hom_{\mathcal G_2}(O_{2\iota(i_1)}, O_{2\iota(j_1)})\widetilde h_p
\end{eqnarray*}
as desired.

To show~(B), we are given $p\in N(i,j)$ and $m\in M_2(j,k)$, and we seek $q\in N(i,k)$ such~that
\begin{center}
\begin{tikzcd}
O_{1i}\ar[dd, "\widetilde h_p"'] \ar[ddr, dashed, "\widetilde h_q"] \\
\\
O_{2j}\ar[r,"f_{2m}"] & O_{2k}
\end{tikzcd}
\end{center}
commutes.
Writing $p=(m',i)$ for some $m'\in M_2(\iota(i),j)$ the above becomes
\begin{center}
\begin{tikzcd}
O_{1i}\ar[d, "h_i"']  \\
O_{2\iota(i)}\ar[d,"f_{2m'}"']\ar[dr, "f_\ell"]\\
O_{2j}\ar[r,"f_{2m}"] & O_{2k}
\end{tikzcd}
\end{center}
where $f_\ell:=f_{2m}f_{2m'}$.
So $q:=(\ell,i)$ will work.
\end{proof}

\begin{corollary}
Equivalent definable connected groupoids in the sense of Hrushovski are isomorphic in $\mathbf {CGpd}_T$.
\end{corollary}

The converse is harder to see directly.
It will follow for finitely faithful definable connected groupoids from the equivalence of categories that we prove below.

We will restrict our attention to the finitely faithful case.
Let $\mathbf {FCGpd}_T$ denote the full subcategory of $\mathbf {CGpd}_T$ generated by the $0$-definable finitely faithful connected groupoids.

\medskip
\subsection{The functor $G:\mathbf {IGI}_T\to \mathbf {FCGpd}_T$}
Given an internal generalised imaginary sort $\mathbb U'=(\mathbb U, S,\dots)$,
Theorem~\ref{groupoid-construction} associated to it the liason groupoid $\mathcal G$ in $\mathbb U$.
Let us write $\mathcal G$ as $G(\mathbb U')$.
Since $G(\mathbb U')$ is a $0$-definable finitely faithful connected groupoid in $\mathbb U$, this makes $G$ a function from the objects of  $\mathbf {IGI}_T$ to the objects of $\mathbf {FCGpd}_T$.
Our aim here is to make $G$ into a covariant functor between these categories.

To do so, we have to also consider the one-object extension of $\mathcal G$ which was called $\mathcal G'$ in Theorem~\ref{groupoid-construction}; it was a definable connected groupoid in $\mathbb U'$ whose additional object was $S$ and which had the property that  $\Aut_{\mathcal G'}(S)=\Aut(\mathbb U'/\mathbb U)$ as groups actiong on $S$.
 We will denote $\mathcal G'$ as $G'(\mathbb U')$.

\begin{remark}
Actually the groupoid associated to the cover $\mathbb U'$ is not
unique; so we are using the axiom of choice to fix a particular one.
Also, the groupoid is really associated to the cover $\Th(\mathbb U')$; so
denoting it by $G(\mathbb U')$ is an abuse of notation.
\end{remark}

Suppose $g : S_1\to S_2$ is a morphism of internal generalised imaginaries $\mathbb U_1,\mathbb U_2$.
Let $\mathcal G_1:= G(\mathbb U_1)$ and $\mathcal G_1':=G'(\mathbb U_1)$.
We write $\mathcal G_1=(O_{1i}, f_{1m})_{i\in I_1, m\in M_1}$ and $\mathcal G_1'=(O_{1i}, f_{1m})_{i\in I_1', m\in M_1'}$ where $I_1'=I_1\cup\{*\}$ and $O_{1*}=S_1$.
We make similar notational conventions to talk about $\mathcal G_2:= G(\mathbb U_2)$ and $\mathcal G_2':=G'(\mathbb U_2)$.
We define $G(g)$ to be the family of functions $O_{1i_1}\to O_{2i_2}$, as $i_1$ ranges in $I_1$ and $i_2$ ranges in~$I_2$, of the form $\alpha_2\circ g\circ\alpha_1$ where $\alpha_1\in\Hom_{\mathcal G_1'}(O_{1i_1}, S_1)$ and  $\alpha_2\in\Hom_{\mathcal G_2'}(S_2,O_{2i_2})$.

It is clear that $G(g)$ is a $0$-definable family of definable function in $(\mathbb U_1\cup\mathbb U_2,g)$.
But as the individual functions are between definable sets in $\mathbb U$, and as $\mathbb U$ is stably embedded in $(\mathbb U_1\cup\mathbb U_2,g)$, we get that the family $G(g)$ is $0$-definable in $\mathbb U$.
Let $N\to I_1\times I_2$ be a $0$-definable parameterising set for $G(g)$.

\begin{lemma}
$G(g)$ is a morphism from $G(\mathbb U_1)$ to $G(\mathbb U_2)$ that depends only on $\operatorname{Th}(\mathbb U_1\cup\mathbb U_2,g)$.
\end{lemma}

\begin{proof}
Suppose $h : S_1 \to S_2$ is such that $ (\mathbb U_1\cup\mathbb U_2,g)\equiv
(\mathbb U_1\cup\mathbb U_2,h)$.
Then, by~\ref{functions-covers}, $h=\gamma g$ for some $\gamma\in \Aut(\mathbb U_2/\mathbb U)$.
By definition of the liason groupoid we have that $\gamma\in\Aut_{\mathcal G_2'}(S_2)$.
Hence, for any $i_2\in I_2$, $\Hom_{\mathcal G_2'}(S_2,O_{2i_2})\gamma=\Hom_{\mathcal G_2'}(S_2,O_{2i_2})$, and so $G(\gamma g)$ yields the same family of functions as $G(g)$.

To show that $G(g)$ is a morphism we need to verify~(A) and~(B).
For~(A), given $p\in N(i_1,i_2)$ and $q\in N(j_1,j_2)$ we will show that 
$$G(g)_{q} \circ \Hom_{\mathcal G_1}(O_{1i_1}, O_{1j_1}) \subseteq \Hom_{\mathcal G_2}(O_{2i_2}, O_{2j_2})\circ G(g)_{p}$$
and the symmetry in the argument will take care of the other containment.
Given $\gamma_1\in \Hom_{\mathcal G_1}$ we seek $\gamma_2\in \Hom_{\mathcal G_2}$ such that
\begin{center}
\begin{tikzcd}
O_{1i_1}\ar[d, "G(g)_p"]\ar[r,"\gamma_1"]&O_{1j_1}\ar[d, "G(g)_q"] \\
O_{2i_2}\ar[r,dashed, "\gamma_2"] & O_{2j_2}
\end{tikzcd}
\end{center}
commutes.
By definition of $G(g)$ this becomes
\begin{center}
\begin{tikzcd}
S_1\ar[d,"g"]&O_{1i_1}\ar[l,"\alpha_1"']\ar[d, "G(g)_p"]\ar[r,"\gamma_1"]&O_{1j_1}\ar[d, "G(g)_q"]\ar[r, "\beta_1"]& S_1\ar[d,"g"]\\
S_2\ar[r, "\alpha_2"']&O_{2i_2}& O_{2j_2}&S_2\ar[l,"\beta_2"]
\end{tikzcd}
\end{center}
Now $\beta_1\gamma_1\alpha_1^{-1}\in\Aut_{\mathcal G_1'}(S_1)=\Aut(\mathbb U_1/\mathbb U)$ and hence extends by stable embeddedness to some $\sigma\in\Aut(\mathbb U_1\cup\mathbb U_2,g)$.
So
\begin{center}
\begin{tikzcd}
S_1\ar[d,"g"]&O_{1i_1}\ar[l,"\alpha_1"']\ar[d, "G(g)_p"]\ar[r,"\gamma_1"]&O_{1j_1}\ar[d, "G(g)_q"]\ar[r, "\beta_1"]& S_1\ar[d,"g"]\\
S_2\ar[rrr, bend right, "\sigma|_{S_2}"]\ar[r, "\alpha_2"']&O_{2i_2}& O_{2j_2}&S_2\ar[l,"\beta_2"]
\end{tikzcd}
\end{center}
commutes.
Hence, $\gamma_2:=\beta_2\sigma|_{S_2}\alpha_2^{-1}$ works.
Note that $\sigma|_{S_2}\in\Aut(\mathbb U_2/\mathbb U)=\Aut_{\mathcal G_2'}(S_2)$ so that $\gamma_2$ is in $\Hom_{\mathcal G_2}$.

For~(B), we are given $p\in N(i,j)$ and $m\in M_2(j,k)$ and seek $q\in N(i,k)$ such~that
\begin{center}
\begin{tikzcd}
O_{1i}\ar[dd, "G(g)_p"'] \ar[ddr, dashed, "G(g)_q"] \\
\\
O_{2j}\ar[r,"f_{2m}"] & O_{2k}
\end{tikzcd}
\end{center}
commutes.
Unravelling, this becomes
\begin{center}
\begin{tikzcd}
S_1\ar[dd,"g"']&O_{1i}\ar[l,"\alpha_1"]\ar[dd, "G(g)_p"']\\
\\
S_2\ar[r,"\alpha_2"]&O_{2j}\ar[r,"f_{2m}"] & O_{2k}
\end{tikzcd}
\end{center}
Since $\beta_2:=f_{2m}\alpha_2\in\Hom_{\mathcal G_2'}(S_2,O_{2k})$, letting $q$ be such that $G(g)_q=\beta_2g\alpha_1$ works.
\end{proof}

\begin{proposition}
$G$ is a covariant functor from $\mathbf {IGI}_T$ to $\mathbf {FCGpd}_T$.
\end{proposition}

\begin{proof}
It remians only to show that if $(\mathbb U_1\cup\mathbb U_2,g)$ and $(\mathbb U_2\cup\mathbb U_3,h)$ are morphisms between internal generalised imaginaries then
$$G(h\circ g)=G(h)\circ G(g).$$
A function in the left-hand-side is of the form $\alpha_3hg\alpha_1$ while one in the right is of the form $\alpha_3h\beta_2\alpha_2g\alpha_1$ -- here $\alpha_1\in\Hom_{G'(\mathbb U_1)}$ with codomain~$S_1$, $\alpha_2\in \Hom_{G'(\mathbb U_2)}$ with domain $S_2$, $\beta_2\in \Hom_{G'(\mathbb U_2)}$ with codomain~$S_2$, $\alpha_3\in\Hom_{G'(\mathbb U_3)}$ with domain $S_3$, and all the compositions make sense.
Given $\alpha_3hg\alpha_1$ on the left-hand-side, to see that it is of the form $\alpha_3h\beta_2\alpha_2g\alpha_1$ we can take any $\alpha_2:S_2\to O_{i_2}$ in $G'(\mathbb U_2)$, for any $i_2\in I_2$, and let $\beta_2:=\alpha_2^{-1}$.

For the converse, suppose we are given $\alpha_3h\beta_2\alpha_2g\alpha_1\in G(h)\circ G(g)$.
Then $\beta_2\alpha_2\in \Aut_{G'(\mathbb U_2)}(S_2)=\Aut(\mathbb U_2/\mathbb U)$ and hence extends by $\mathbb U_2$ being stably embedded in $(\mathbb U_1\cup\mathbb U_2,g)$ to an automorphism  $\sigma$ of the latter.
So $\beta_2\alpha_2g=g\sigma|_{S_1}$.
Hence $\alpha_3h\beta_2\alpha_2g\alpha_1=\alpha_3hg\sigma|_{S_1}\alpha_1$.
Since $\sigma|_{S_1}\in\Aut(\mathbb U_1/\mathbb U)=\Aut_{G'(\mathbb U_1)}(S_1)$, we have that $\sigma|_{S_1}\alpha_1\in\Hom_{G'(\mathbb U_1)}$ with codomain~$S_1$, so $\alpha_3hg\sigma|_{S_1}\alpha_1$ has the desired form to be in $G(h\circ g)$.
\end{proof}

\medskip
\subsection{The functor $C:\mathbf {FCGpd}_T\to\mathbf {IGI}_T$}
\label{Cfunctor}
Recall that Theorems~\ref{cover-construction} and~\ref{cover-construction-refined} associated to any $0$-definable finitely faithful connected groupoid $\mathcal G$ in $\mathbb U$ an internal generalised imaginary sort denoted there by $\mathbb U''=(\mathbb U,O_*,\dots)$.
Let us relabel this cover as $C(\mathcal G)$, but retain $O_*$ as the name of the new sort.
So $C$ is a function from the objects of $\mathbf {FCGpd}_T$ to the objects of $\mathbf {IGI}_T$, which we now make into a covariant functor.

Let us first recall that we have a one-object extension of $\mathcal G$ denoted by $\mathcal G'$, which is a definable connected groupoid in $C(\mathcal G)$ whose new object is $O_{*}$, and with the property that $\Aut_{\mathcal G'}(O_*)=\Aut\big(C(\mathcal G)/\mathbb U)$ as groups acting on $O_*$.

Now suppose we have two $0$-definable finitely faithful connected groupoids in~$\mathbb U$, $\mathcal G_1 = (O_{1i}, f_{1m})_{i \in I_1, m \in M_1}$ and $\mathcal G_2 = (O_{2i}, f_{2m})_{i \in I_2, m \in M_2}$, and $h : \mathcal G_1 \to \mathcal G_2$ is a $0$-definable morphism between them.
We obtain covers $C(\mathcal G_1)=(\mathbb U,O_{1*},\dots)$ and $C(\mathcal G_2)=(\mathbb U,O_{2*},\dots)$.
Recall from the construction on page~\pageref{constructiondetails} that
$$O_{1*}=O_{1i_1}\times\{0\}$$
for some fixed $i_1\in I_1$, and
$$O_{2*}=O_{2i_2}\times\{0\}$$
for some fixed $i_2\in I_2$.
Let $p\in N(i_1,i_2)$ be arbitrary so that $h_p:O_{1i_1}\to O_{2i_2}$, and define $C(h):=h_p\times\{0\}:O_{1*}\to O_{2*}$.

\begin{lemma}
\label{c(h)}
$C(h)$ determines a morphism from $C(\mathcal G_1)$ to $C(\mathcal G_2)$ that depends only on $h$ and not on the choices of $i_1,i_2,p$.
\end{lemma}

\begin{proof}
First let us check that $C(h)$, or rather $\operatorname{Th} \big(C(\mathcal G_1)\cup C(\mathcal G_2), C(h)\big)$, is well defined.
So let $j_1 \in I_1$, $j_2 \in I_2$ and $q\in N(j_1,j_2)$ be different choices.
Let $\overline{C(\mathcal G_1)}$ denote the internal cover constructed exactly as $C(\mathcal G_1)$ but using $j_1$ instead of $i_1$.
We already know by Theorem~\ref{cover-construction} that $C(\mathcal G_1)$ and $\overline{C(\mathcal G_1)}$ are isomorphic over $\mathbb U$; in fact such an isomorphism is induced by fixing any $\alpha_1 \in\Hom_{\mathcal G_1}(O_{i_1}, O_{j_1})$.
By~(A), there is $\alpha_2 \in \Hom_{\mathcal G_2}(O_{2i_2}, O_{2j_2})$ such that $h_q\alpha_1=\alpha_2h_p$.
Now, $\alpha_2$ induces an isomorphism of $C(\mathcal G_2)$ and $\overline{C(\mathcal G_2)}$ over $\mathbb U$.
The commutativity condition above implies that $\alpha_1 \cup \alpha_2$ induces an isomorphism from $\big(C(\mathcal G_1)\cup C(\mathcal G_2), h_p\times\{0\}\big)$ to $\big(\overline{C(\mathcal G_1)}\cup \overline{C(\mathcal G_2)}, h_q\times\{0\}\big)$.

Toward proving that $\big(C(\mathcal G_1)\cup C(\mathcal G_2), C(h)\big)$ is saturated, let us first recall from the construction that we have
\begin{center}
\begin{tikzcd}
O_{1i_1}\ar[r, "h_p"]\ar[d, "\alpha_1"'] & O_{2i_2} \ar[d, "\alpha_2"]\\
O_{1*}\ar[r, "C(h)"]& O_{2*}
\end{tikzcd}
\end{center}
for some $\alpha_1\in\Hom_{\mathcal G_1'}(O_{1i_1},O_{1*})$ and $\alpha_2\in\Hom_{\mathcal G_2'}(O_{2i_2},O_{2*})$.
Indeed, $\alpha_1=\id_{O_{1i_1}}\times\{0\}$ and $\alpha_2=\id_{O_{2i_2}}\times\{0\}$, though of course this particular representation is not known to the theory of the morphism $C(h)$.

Now let $\big(\overline{C(\mathcal G_1)}\cup \overline{C(\mathcal G_2)}, \overline{C(h)}\big)$ be a saturated model of $\operatorname{Th}\big(C(\mathcal G_1)\cup C(\mathcal G_2), C(h)\big)$.
By Lemma~\ref{union} we know that $C(\mathcal G_1)\cup C(\mathcal G_2)$ is a saturated model of the complete theory $\operatorname{Th}\big(C(\mathcal G_1)\big)\cup \operatorname{Th}\big(C(\mathcal G_2)\big)$, and so we have an isomorphism
$$\sigma_1\cup\sigma_2:C(\mathcal G_1)\cup C(\mathcal G_2)\to\overline{C(\mathcal G_1)}\cup \overline{C(\mathcal G_2)}.$$
We may as well identify the copies of $\mathbb U$ in both so that $\sigma_1\cup\sigma_2$ is the identity on~$\mathbb U$.
On the other hand, the existence of the above commuting square is part of the theory of the morphism $C(h)$, and so we have
\begin{center}
\begin{tikzcd}
O_{1i_1}\ar[r, "h_n"]\ar[d, "\overline \alpha_1"'] & O_{2i_2} \ar[d, "\overline \alpha_2"]\\
\overline O_{1*}\ar[r, "\overline{C(h)}"]& \overline O_{2*}
\end{tikzcd}
\end{center}
for some $\overline \alpha_1\in\Hom_{\overline{\mathcal G_1'}}(O_{1i_1},\overline O_{1*})$ and $\overline \alpha_2\in\Hom_{\overline{\mathcal G_2'}}(O_{2i_2},\overline O_{2*})$.
Since $\sigma_1^{-1}$ is an isomorphism from $\overline{C(\mathcal G_1)}$ to $C(\mathcal G_1)$ that is the identity on $\mathbb U$, and $\overline\alpha_1\in\Hom_{\overline{\mathcal G_1'}}$ has its domain in $\mathbb U$, we get that
$\sigma^{-1}\overline\alpha_1\in \Hom_{\mathcal G_1'}$.
Hence, so is $\overline \alpha_1^{-1}\sigma_1$.
It follows that
$\widetilde \alpha_1:=\overline \alpha_1^{-1}\sigma_1\alpha_1\in\Aut_{\mathcal G_1}(O_{1i_1})$.
By~(A), there is $\widetilde \alpha_2\in \Aut_{\mathcal G_2}(O_{2i_2})$ such that $h_p\widetilde\alpha_1=\widetilde\alpha_2h_p$.
Now, since $\sigma_2^{-1}$ is an isomorphism from $\overline{C(\mathcal G_2)}$ to $C(\mathcal G_2)$ that is the identity $\mathbb U$, we have that $\beta:=\sigma_2^{-1}\overline \alpha_2\widetilde \alpha_2\alpha_2^{-1}\in\Aut_{\mathcal G_2'}(O_{2*})=\Aut\big(C(\mathcal G_2)/\mathbb U)$.
A final diagram chase, shows that if we replace $\sigma_2$ with $\sigma_2\beta$ then
\begin{center}
\begin{tikzcd}
O_{1*}\ar[r, "C(h)"]\ar[d, "\sigma_1"'] & O_{2*} \ar[d, "\sigma_2"]\\
\overline O_{1*}\ar[r, "\overline{C(h)}"] & \overline O_{2*}
\end{tikzcd}
\end{center}
commutes.
That is,
$$\sigma_1\cup\sigma_2:\big(C(\mathcal G_1)\cup C(\mathcal G_2), C(h)\big)\to\big(\overline{C(\mathcal G_1)}\cup \overline{C(\mathcal G_2)}, \overline{C(h)}\big)$$
is an isomorphism, and so $\big(C(\mathcal G_1)\cup C(\mathcal G_2), C(h)\big)$ is saturated.

Note that $\mathbb U$ is stably embedded in $\big(C(\mathcal G_1)\cup C(\mathcal G_2), C(h)\big)$.
Indeed, if $\sigma \in \Aut(\mathbb U)$ then it induces an isomorphism between $\big(C(\mathcal G_1)\cup C(\mathcal G_2), C(h)\big)$ and the analogous construction with $\sigma(i_1),\sigma(i_2),\sigma(p)$ rather than $i_1,i_2,p$.
But by the first paragraph of this proof, this is isomorphic to $\big(C(\mathcal G_1)\cup C(\mathcal G_2), C(h)\big)$ over $\mathbb U$.
Composing these isomorphism we get an automorphism of $\big(C(\mathcal G_1)\cup C(\mathcal G_2), C(h)\big)$ that extends $\sigma$.

Finally, we show that $C(\mathcal G_1)$ is stably embedded in $\big(C(\mathcal G_1)\cup C(\mathcal G_2), C(h)\big)$.
Suppose $\sigma\in\Aut\big(C(\mathcal G_1)\big)$.
By the previous paragraph it suffices to consider $\sigma\in \Aut\big(C(\mathcal G_1)/\mathbb U\big)= \Aut_{\mathcal G_1'}(O_{1*})$.
So $\sigma$ is induced by some $\alpha_1\in \Aut_{\mathcal G_1}(O_{1i_1})$.
By~(A), we can find $\alpha_2 \in \Aut_{\mathcal
G_2}(O_{2i_2})$ such that $h_p\alpha_1=\alpha_2h_p$.
Now, $\alpha_2$ induces some $\sigma_2\in\Aut_{\mathcal G_2'}(O_{2*})=\Aut\big(C(\mathcal G_2)/\mathbb U\big)$.
The above commuting property ensures that $\sigma\cup\sigma_2$ is an automorphism of $\big(C(\mathcal G_1)\cup C(\mathcal G_2), C(h)\big)$.
A similar argument shows that $C(\mathcal G_2)$ is stably embedded in $\big(C(\mathcal G_1)\cup C(\mathcal G_2), C(h)\big)$.
\end{proof}

\begin{proposition}
$C$ is a covariant functor from $\mathbf {FCGpd}_T$ to $\mathbf {IGI}_T$.
\end{proposition}

\begin{proof}
It remains to show that if $g:\mathcal G_1\to\mathcal G_2$ and $h:\mathcal G_2\to\mathcal G_3$ are morphisms of definable connected groupoids, then $C(h\circ g)=C(h)\circ C(g)$.
Using $i_1,i_2,p$ for $C(g)$ and $i_2,i_3,q$ for $C(h)$ we get that the right-hand-side is given by $(h_q\times\{0\})\circ(g_p\times\{0\})$ where $g_p:O_{1i_1}\to O_{2i_2}$ and $h_q:O_{2i_2}\to O_{2i_3}$.
But now we can use $i_1,i_3, n=(p,q)$ so that the left-hand-side is given by $(h\circ g)_n\times\{0\}$.
As $(h\circ g)_n=h_q\circ g_p$ by definition, these agree.
\end{proof}

\medskip
\subsection{The equivalence}

\begin{theorem}
\label{equivcat}
The functors $G$ and $C$ defined above establish an equivalence between the categories $\mathbf {IGI}_T$ and $\mathbf {FCGpd}_T$.
\end{theorem}

\begin{proof}
We first construct a natural isomorphism of functors $\eta:\id_{\mathbf{IGI}_T} \to CG$.
That is, for each internal generalised imaginary sort $\mathbb V$ of $\mathbb U$ (the reader will see in a minute why for the purposes of this proof we use $\mathbb V$ rather the usual terminology of $\mathbb U'$) we construct an isomorphism of covers 
$\eta_{\mathbb V}:\mathbb V\to CG(\mathbb V)$ such that 
\begin{center}
\begin{tikzcd}
\mathbb U_1\ar[r, "\eta_{\mathbb U_1}"]\ar[d, "g"] & CG(\mathbb U_1) \ar[d, "CG(g)"]\\
\mathbb U_2\ar[r, "\eta_{\mathbb U_2}"] & CG(\mathbb U_2)\\
\end{tikzcd}
\end{center}
commutes, for every morphism $g:\mathbb U_1\to\mathbb U_2$ in $\mathbf {IGI}_T$.

Given $\mathbb V=(\mathbb U, S,\dots)$ let $\mathcal G:=G(\mathbb V)=(O_i,f_m)_{i\in I, m\in M}$ be the liason groupoid in $\mathbb U$.
Then $CG(\mathbb V)=C(\mathcal G)$ is another cover of $\mathbb U$ whose new sort  is $O_*:=O_i\times\{0\}$ for some fixed $i\in I$.
We have two one-object groupoid extensions of $\mathcal G$:
in $\mathbb V$ we have $G'(\mathbb V)$ whose extra object is $S$, whereas in $CG(\mathbb V)$ we have $G(\mathbb V)'$ whose additional object is $O_*$.
Now, fix $\alpha\in \Hom_{G'(\mathbb V)}(S,O_i)$, and define
$$\eta_{\mathbb V}:=\alpha\times\{0\}:S\to O_*,$$
where we mean here the function given by $x\mapsto (\alpha(x),0)$.
That $\eta_{\mathbb V}$ does not depend on the choice of $i$ and $\alpha$ follows from by now familiar methods using that $\Aut_{G'(\mathbb V)}(S)=\Aut(\mathbb V/\mathbb U)$.
Saturation of $(\mathbb V\cup CG(\mathbb V), \eta_{\mathbb V})$ follows very much in the same way as the saturation argument in the proof of Lemma~\ref{c(h)}, indeed somewhat more simply, using this time that $\Aut_{G(\mathbb V)'}(O_*)=\Aut(CG(\mathbb V)/\mathbb U)$.
It is similarly not hard to deduce, following the methods of the proof of Lemma~\ref{c(h)}, that both $\mathbb V$ and $CG(\mathbb V)$ are stably embedded in $(\mathbb V\cup CG(\mathbb V), \eta_{\mathbb V})$.
This shows that $\eta_{\mathbb V}$ is indeed a morphism of covers.
It is an isomorphism because it is bijective.

We now check that the above square commutes for any morphism $g:S_1\to S_2$ of internal generalised imaginaries $\mathbb U_1=(\mathbb U,S_1,\dots)$ and $\mathbb U_2=(\mathbb U,S_2,\dots)$.
Let $G(\mathbb U_1)=(O_{1i},f_{1m})_{i\in I_1, m\in M_1}$ and fix $i_1\in I_1$ so that the new sort of $CG(\mathbb U_1)$ is $O_{1*}:=O_{1i_1}\times\{0\}$.
Let $\alpha_1\in\Hom_{G'(\mathbb U_1)}(S_1,O_{1i_1})$ be such that $\eta_{\mathbb U_1}=\alpha_1\times\{0\}$.
Make similar notational conventions and choices for $G(\mathbb U_2)$, $CG(\mathbb U_2)$, and~$\eta_{\mathbb U_2}$.
Thus $\eta_{\mathbb U_2} g$ is the theory of $(\mathbb U_1\cup CG(\mathbb U_2), \alpha_2 g\times\{0\})$.
But now
note that by definition of how the functor $G$ acts on morphisms we have that $\alpha_2g\alpha_1^{-1}: O_{1i_1}\to O_{2i_2}$ is one of the functions in $G(g)$.
Thus, by how $C$ was defined to act on morphisms, $CG(g)$ is the theory of $(CG(\mathbb U_1)\cup CG(\mathbb U_2), \alpha_2g\alpha_1^{-1})$.
Hence $CG(g)\eta_{\mathbb U_1}$ is also the theory of $(\mathbb U_1\cup CG(\mathbb U_2), \alpha_2 g\times\{0\})$, as desired.

So we have a a natural isomorphism $\eta:\id_{\mathbf{IGI}_T} \to CG$.

Next we construct a natural isomorphism $\varepsilon:\id_{\mathbf {FCGpd}_T}\to GC$.

Let $\mathcal G = (O_{1i}, f_{1m})_{i \in I_1, m \in M_1}$
be a $0$-definable finitely faithful connected groupoid in $\mathbb U$.
So $C(\mathcal G)$ is an internal generalised imaginary with new sort $O_*=O_{i}\times\{0\}$ for some fixed $i\in I_1$.
We get a liason groupoid $GC(\mathcal G)=(O_{2i}, f_{2m})_{i \in I_2, m \in M_2}$ in $\mathbb U$, and we wish to define a natural isomorphism $\varepsilon_\mathcal G:\mathcal G\to GC(\mathcal G)$.

Recall that  $\mathcal G'$ is a definable connected groupoid in $C(\mathcal G)$ extending $\mathcal G$ by the one new object $O_*$.
On the other hand, $G'C(\mathcal G)$ is also a definable connected groupoid in $C(\mathcal G)$, this time extending $GC(\mathcal G)$ but also by the one new object $O_*$.
Note that
$$\Aut_{\mathcal G'}(O_*)=\Aut\big(C(\mathcal G)/\mathbb U\big)=\Aut_{G'C(\mathcal G)}(O_*)$$
as groups acting on $O_*$.
This suggests how we should define $\varepsilon_\mathcal G:\mathcal G\to GC(\mathcal G)$; given $i_1\in I_1$ and $i_2\in I_2$, the set of function from $O_{1i_1}$ to $O_{2i_2}$ in $\varepsilon_\mathcal G$ is
$$\Hom_{G'C(\mathcal G)}(O_*,O_{2i_2})\circ\Hom_{\mathcal G'}(O_{1i_1},O_*).$$
As both $\mathcal G'$ and $G'C(\mathcal G)$ are $0$-definable in $C(\mathcal G)$, this is clearly uniformly $0$-definable in $C(\mathcal G)$, and hence by stable embeddability in $\mathbb U$.
Let $N\to I_1\times I_2$ parametrise $\varepsilon_\mathcal G$.

To see that $\varepsilon_\mathcal G$ does not depend on the choice of $i$, suppose $\overline{C(\mathcal G)}$ is an isomorphic copy of $C(\mathcal G)$ over $\mathbb U$, which is exactly what a different choice of $i$ would produce.
Suppose $\sigma$ is the isomorphism.
We need to show that
$$\Hom_{G'C(\mathcal G)}(O_*,O_{2i_2})\circ\Hom_{\mathcal G'}(O_{1i_1},O_*)= \Hom_{G'\overline{C(\mathcal G)}}(\overline O_*,O_{2i_2})\circ\Hom_{\overline{\mathcal G'}}(O_{1i_1},\overline O_*).$$
But  $\sigma$ would take $\Hom_{G'C(\mathcal G)}(O_*,O_{2i_2})$ to $\Hom_{G'\overline{C(\mathcal G)}}(\overline O_*,O_{2i_2})$, and $\Hom_{\mathcal G'}(O_{1i_1},O_*)$ to $\Hom_{\overline{\mathcal G'}}(O_{1i_1},\overline O_*)$, so that it takes the left-hand-side of the desired equality to the right-hand-side.
But as $\sigma$ is the identity on $\mathbb U$, these sets must be equal.

Next we check condition~(A) of being a morphism.
Fix $p\in N(i_1,i_2)$ and $q\in N(j_1,j_2)$.
Write $(\varepsilon_\mathcal G)_p=\beta\alpha$ where $\alpha\in \Hom_{\mathcal G'}(O_{1i_1},O_*)$ and $\beta\in\Hom_{G'C(\mathcal G)}(O_*,O_{2i_2})$.
Similarly, write $(\varepsilon_\mathcal G)_q=\delta\gamma$ where $\gamma\in \Hom_{\mathcal G'}(O_{1j_1},O_*)$ and $\delta\in\Hom_{G'C(\mathcal G)}(O_*,O_{2j_2})$.
Then
\begin{eqnarray*}
(\varepsilon_\mathcal G)_q\Hom_{\mathcal G'}(O_{1i_1},O_{1j_1})
&=&
\delta\gamma\Hom_{\mathcal G'}(O_{1i_1},O_{1j_1})\\
&=&
\delta\Aut_{\mathcal G'}(O_*)\alpha\\
&=&
\delta\Aut_{G'C(\mathcal G)}(O_*)\alpha\\
&=&
\Hom_{G'C(\mathcal G)}(O_{2i_2},O_{2j_2})\beta\alpha\\
&=&
\Hom_{G'C(\mathcal G)}(O_{2i_2},O_{2j_2})(\varepsilon_\mathcal G)_p
\end{eqnarray*}
as desired.

That condition~(B) holds follows easily from the definitions and the fact that $\Hom_{G'C(\mathcal G)}$ is closed under composition.
So $\varepsilon_{\mathcal G}$ is a morphism.
It is an isomorphism by Lemma~\ref{iso=bij} since it is made up of bijections.

Finally, it remains to show that for any morphism $h:\mathcal G_1\to\mathcal G_2$,
\begin{center}
\begin{tikzcd}
\mathcal G_1\ar[r, "\varepsilon_{\mathcal G_1}"]\ar[d, "h"] & GC(\mathcal G_1) \ar[d, "GC(h)"]\\
\mathcal G_2\ar[r, "\varepsilon_{\mathcal G_2}"] & GC(\mathcal G_2)\\
\end{tikzcd}
\end{center}
commutes.
We leave this to the reader as it is notationaly rather tedious to write down and does not use any observation that we have not already made, but rather requires simply an unravelling of the definitions.
\end{proof}

\bigskip
\section {Uniformity}
\label{section-uniformity}

\noindent
We now investigate the extent to which these functors $G$ and $C$ preserve uniform definability in families.
As usual we fix a monster model $\mathbb U$ of our complete (multi-sorted) $L$-theory $T$, which we assume admits elimination of imaginaries.

First, let us be precise about what we mean by families of connected groupoids.

\begin{definition}
Given a $0$-definable set $A$ in $\mathbb U$, by a {\em $0$-definable family of connected groupoids over $A$}, or a {\em $0$-definable connected groupoid relative to $A$}, denoted by $\mathcal G\to A$, we mean the following:
a $0$-definable set $I$ equipped with a surjective $0$-definable function $I\to A$, a $0$-definable set $M$ equipped with a surjective $0$-definable function $\displaystyle M\to I\times_A I$, a $0$-definable set $O$ with a surjective $0$-definable $O \to I$, and a $0$-definable subset $\displaystyle f \subseteq M\times_{I\times_A I} (O \times_A O)$ such that for each $a\in A$, the fibres $I_a, M_a\to I_a\times I_a, O_a\to I_a$, and $f_a\subseteq M_a\times_{I_a\times I_a}(O_a\times O_a)$, form an $a$-definable connected groupoid $\mathcal G_a=\big((O_a)_i,(f_m)_i\big)_{i\in I_a,m\in M_a}$.
\end{definition}

Note that this is really nothing other than a $0$-definable groupoid in $\mathbb U$ that may not be connected.
Indeed, the total space $\mathcal G=(O_i,f_m)_{i\in I, m\in M}$ is a $0$-definable groupoid whose connected components are precisely the $\mathcal G_a$ for $a\in A$.
Conversely, if we start with a possibly disconnected $0$-definable groupoid $\mathcal G=(O_i,f_m)_{i\in I, m\in M}$ in $\mathbb U$, and we let (following~\cite{hrushovski}) $\operatorname{Iso}_{\mathcal G}$ denote the ($0$-definable) equivalence relation on $I$ where $(i,j)\in \operatorname{Iso}_{\mathcal G}$ if and only if $M(i,j)\neq\emptyset$, then $\mathcal G$ naturally obtains the structure of a definable connected groupoid relative to $A:=I/\operatorname{Iso}_{\mathcal G}$.
So asking about definable families of connected groupoids amounts to generalising from connected groupoids to possible disconnected groupoids.

On the other side, what should we mean by an internal generalised imaginary sort {\em relative to $A$}?
This is what Hrushovski~\cite{hrushovski} refers to as a $1$-analysable cover over $A$, and we will continue to use his terminology.

\begin{definition}
Given a $0$-definable set $A$ in $\mathbb U$,
a {\em $1$-analysable cover over $A$} is a cover $\mathbb U'=(\mathbb U,S,\dots)$ of $\mathbb U$ by a single new sort, in a language that is finitely generated over $L$, equipped with a $0$-definable surjective map $S\to A$ such that each fibre $S_a$ is internal to $\mathbb U$.
\end{definition}

As Hrushovski mentions, in the Remark following Lemma~3.1 of~\cite{hrushovski}, the $1$-analysable covers that arise have the additional property that there are no interesting definable relations between the distinct fibres of $S\to A$.
We formalise this as follows.

\begin{definition}
We will say that a cover $\mathbb U'=(\mathbb U,S,\dots)$ by a single new sort equipped with a $0$-definable surjective map $S\to A$, where $A$ is a $0$-definable set in~$\mathbb U$, has {\em independent fibres} if
$$\Aut(\mathbb U'/\mathbb U)=\prod_{a\in A}\Aut(S_a/\mathbb U)$$
in the sense that $\sigma\mapsto(\sigma|_{S_a}:a\in A)$ is an isomorphism.
Here,
$$\Aut(S_a/\mathbb U):=\{\sigma|_{S_a}:\sigma\in\Aut\big(\mathbb U'/\mathbb U\big)\}.$$
\end{definition}

The notation $\Aut(S_a/\mathbb U)$ could be misleading, it could be confused with the {\em a~priori} larger group of bijections $f:S_a\to S_a$ such that $f\cup\id_{\mathbb U}$ is a partial elementary map of $\mathbb U'$.
The following lemma implies that in the case of independent fibres these groups do actually agree.

\begin{lemma}
\label{sastable}
Suppose $A$ is $0$-definable in $\mathbb U$, and $\mathbb U'$ is a cover of $\mathbb U$ by a single new sort $S$ equipped with a $0$-definable surjective map $S\to A$ with independent fibres.
Then $(\mathbb U,S_a)$ is stably embedded in $\mathbb U'$ for each $a\in A$.
\end{lemma}

\begin{proof}
It suffices to show that every automorphism of the structure induced on $(\mathbb U, S_a)$ by the $a$-definable sets in $\mathbb U'$ extends to an automorphism of $\mathbb U'$.
Let $\rho$ be such an automorphism of the induced structure on $(\mathbb U,S_a)$.
Since $\mathbb U$ is stably embedded in $\mathbb U'$ -- this is part of being a cover -- $\rho|_{\mathbb U}$ does extend to $\Aut(\mathbb U')$, and hence, by precomposing with the inverse of this extension, we may assume that $\rho|_{\mathbb U}=\id_{\mathbb U}$.
We will show that $\rho$ extends by identity to an automorphism of $\mathbb U'$.
That is, we will show $\tp(bc)=\tp\big(\rho(b)c\big)$ for any finite tuple $b$ from $S_a$ and $c$ from $\mathbb U'\setminus S_a$.
We know that $\tp(b/\mathbb U)=\tp\big(\rho(b)/\mathbb U\big)$.
As $\mathbb U$ is stably embedded in $\mathbb U'$ there exists $\sigma\in\Aut(\mathbb U'/\mathbb U)$ with $\sigma(b)=\rho(b)$.
Now $\sigma|_{S_a}\in\Aut(S_a/\mathbb U)$, and so by independence of fibres there is $\tau\in\Aut(\mathbb U'/\mathbb U)$ that agrees with $\sigma$ on $S_a$ and is the identity everywhere else.
Hence, for any finite tuple $c$ from $\mathbb U'\setminus S_a$, we get
$$\tp(bc)=\tp\big(\tau(bc)\big)=\tp\big(\sigma(b)c\big)=\tp\big(\rho(b)c)$$
as desired.
\end{proof}

We are now ready to show that the liason groupoid functor $G$ relativises -- that the construction of Theorem~\ref{groupoid-construction} is uniformly definable in families -- under the assumption that the fibres are independent.

\begin{theorem}
\label{relative-groupoid-construction}
Suppose $A$ is $0$-definable set in $\mathbb U$, and $\mathbb U' = (\mathbb U, S,\dots)$ is a $1$-analysable cover over $A$ with independent fibres.
Then there are $0$-definable families of finitely faithful connected groupoids, $\mathcal G\to A$ in $\mathbb U$ and $\mathcal G'\to A$ in $(\mathbb U')^{\eq}$, such that for each $a\in A$
\begin{itemize}
\item
$\mathcal G_a'$ is a one-object extension of $\mathcal G_a$ with that object being $S_a$, and
\item
$\Aut_{\mathcal G_a'}(S_a)$ together with its action on $S_a$ is isomorphic to
$\Aut(S_a/\mathbb U)$.
\end{itemize}
We denote $\mathcal G\to A$ by $G_A(\mathbb U')$.
\end{theorem}

\begin{proof}
The first step will be to produce for each $a\in A$ an $a$-definable finitely faithful connected groupoid $\mathcal G_a$ in $\mathbb U$ with a one-object extension $\mathcal G'_a$ in $(\mathbb U')^{\eq}$ such that $\Aut_{\mathcal G'_a}(S_a) = \Aut(S_a/\mathbb U)$.
We will then show that these vary uniformly with the parameter $a$.

Fix $a \in A$, and let $M$ be the structure induced on $(\mathbb U,S_a)$ by the $a$-definable sets in~$\mathbb U'$.
Then $M$ is an internal cover of $\mathbb U$.
We claim that, because of the independnece of fibres, the language of $M$ is finitely generated over that of~$\mathbb U$.
Indeed, the proof of Proposition~\ref{finite-language} actually shows that $M$ has a finitely generated language if for any definable family of (graphs of) definable bijections from $S_a$ to itself in $M$, say $F\subseteq B\times S_a^2$, the set $\{b\in B:F_b\cup\id_{\mathbb U}\in\Aut(M/\mathbb U)\}$ is definable.
By stable embeddeness (Lemma~\ref{sastable}), $\Aut(M/\mathbb U)=\Aut(S_a/\mathbb U)$.
By independence of fibres, $F_b\cup\id_{\mathbb U}\in\Aut(S_a/\mathbb U)$ if and only if $\displaystyle F_b\cup \id_{\mathbb U'\setminus S_a}\in\Aut(\mathbb U')$.
And the latter is a definable (in $\mathbb U'$) property of the parameter $b$ as the language of $\mathbb U'$ is finitely generated over that of $\mathbb U$.
Hence, by stable embeddeness of $(\mathbb U, S_a)$ in $\mathbb U'$ again, it is a definable property of $b$ in $M$.
So $M$ is an internal generalised imaginary sort of~$\mathbb U$, and we can apply Theorem~\ref{groupoid-construction} to $M$.
We obtain an $a$-definable connected liason groupoid $\mathcal G_a$ in $\mathbb U$ and its one-object connected extension $\mathcal G'_a$ in $M$ with new object~$S_a$.
So $\mathcal G_a'$ is $a$-definable in $\mathbb U'$.
Moreover,  $\Aut_{\mathcal G'_a}(S_a) = \Aut(M/\mathbb U)=\Aut(S_a/\mathbb U)$.

Note that despite the notation, we do not yet have uniformity as $a$ varies.
Our remaining goal is to make these groupoids part of uniformly $0$-definable families of connected groupoids $\mathcal G\to A$ in $\mathbb U$ and $\mathcal G'\to A$ in $(\mathbb U')^{\eq}$.

Fix $a\in A$ and formulas (involving the parameter $a$) that define all the data that make up $\mathcal G_a$ and $\mathcal G_a'$.
By compactness it suffices to produce a $0$-definable set $a\in U(a)\subseteq A$ such that for all $b\in U(a)$, when one replaces $a$ by $b$ in these formulas one still gets connected groupoids $\mathcal G_b$ and $\mathcal G_b'$, where $\mathcal G'_b$ is a one-object extension of $\mathcal G_b$ with new object $S_b$, and such that $\Aut_{\mathcal G'_b}(S_b)=\Aut(S_b/\mathbb U)$.
The only part of this that is not clear is that the condition $\Aut_{\mathcal G'_b}(S_b)=\Aut(S_b/\mathbb U)$ can be $0$-definably imposed on the parameter $b$.

That $\Aut(S_b/\mathbb U)\subseteq \Aut_{\mathcal G'_b}(S_b)$ is no additional condition at all on $b$, it holds for all $b$ such that $\mathcal G_b'$ is one-object extension of $\mathcal G_b$ with new object $S_b$.
Indeed, let $\sigma\in \Aut(\mathbb U'/\mathbb U)$.
Choose $f_m\in\Hom_{\mathcal G_b'}(S_b, O_i)$ for some object $O_i$ from $\mathcal G_b$.
Then for all $x\in S_b$ we have
\begin{eqnarray*}
f_m(x)
&=&
\sigma(f_m(x))\ \ \ \text{ as $f_m(x)\in O_i\subseteq \mathbb U$}\\
&=&
f_{\sigma(m)}(\sigma(x))
\end{eqnarray*}
and $f_{\sigma(m)}\in\Hom_{\mathcal G_b'}(S_b, O_i)$ as $\sigma(b)=b$.
Hence $\sigma|_{S_b}=f_{\sigma(m)}^{-1}f_m\in\Aut_{\mathcal G_b'}(S_b)$.

So it remains to find a formula $\theta(z)$ satisfied by $a$ and implying that every element of $\Aut_{\mathcal G_z'}(S_z)$ extends to an element of $\Aut(\mathbb U'/\mathbb U)$.
By finite generatedness of the language over $L$, there is $\theta(z)$ expressing that every element of $\Aut_{\mathcal G_z'}(S_z)$ when extended by the identity is an element of $\Aut(\mathbb U'/\mathbb U)$.
That this holds of $a$ is by independence of the fibres plus the fact that $\Aut_{\mathcal G'_a}(S_a) = \Aut(S_a/\mathbb U)$.
\end{proof}

Next we relativise the functor $C$.

\begin{theorem}
\label{relative-cover-construction}
Suppose $\mathcal G\to A$ is a finitely faithful $0$-definable connected groupoid relative to $A$ in $\mathbb U$.
There exists a $1$-analysable cover over $A$, $C_A(\mathcal G)=(\mathbb U,S,\dots)$,  with independent fibres, and a $0$-definable connected groupoid relative to $A$, $\mathcal G'\to A$, in~$C_A(\mathcal G)$, such that for each $a\in A$,
\begin{itemize}
\item
$\mathcal G_a'$ is a one-object extension of $\mathcal G_a$, the new object being $S_a$, and
\item
$\Aut_{\mathcal G_a'}(S_a)$ together with its action on $S_a$ is isomorphic to 
$\Aut(S_a/\mathbb U)$.
\end{itemize}
\end{theorem}

\begin{proof}
As in the proof of Theorem~\ref{cover-construction-refined} this goes by first following Hrushovski~\cite{hrushovski} in constructing $\mathcal G'\to A$ in a cover $\mathbb U'$ that involves {\em two} new sorts, and then using the finite faithfulness assumption to get rid of one of them.
In addition one has to establish the finite-generatedness of the language (part of being a $1$-analysable cover) and the independence of the fibres.

The first step is a straightforward relativisation of the construction given on page~\pageref{constructiondetails}.
This is actually done in~\cite[Lemma~3.1]{hrushovski} where Hrushovski does not assume the definable groupoid is connected.
Instead of fixing a single $i\in I$, we fix a section $\nu:A\to I$ and let $S_a:=O_{\nu(a)}\times\{a\}$.
We can then construct $\mathcal G'\to A$ so that for each $a\in A$ and $i\in I_a$, $\Hom_{\mathcal G'}(S_a, O_i)$ is a copy of $\Hom_{\mathcal G}(O_{\nu(a)}, O_i)$, $\Hom_{\mathcal G'}(O_i,S_a)$ is a copy of $\Hom_{\mathcal G}(O_i,O_{\nu(a)})$, and $\Aut_{\mathcal G'}(S_a)$ is a copy of $\Aut_{\mathcal G}(O_{\nu(a)})$.
Hrushovski proves that this groupoid lives $0$-definably in a cover $\mathbb U'$ of $\mathbb U$ with  new sorts $S=\bigcup_{a\in A}S_a$ and $M_*$, the latter equipped with a map to $A$ such that $(M_*)_a$ encodes the morphisms that involve $S_a$ as either domain or codomain or both.
That
$\Aut_{\mathcal G_a'}(S_a)=\{\sigma|_{S_a}:\sigma\in\Aut(\mathbb U')\text{ and }\sigma|_{\mathbb U}=\id_{\mathbb U}\}$, is part of~\cite[Lemma~3.1]{hrushovski}.

For the second step, since $\mathcal G$ is finitely faithful, the proof of Proposition~\ref{whenff} shows that $M_*\subseteq\dcl(\mathbb U, S)$, and so $(\mathbb U,S)$ is stably embedded in $\mathbb U'$. We define $C_A(\mathcal G)$ to be the structure induced on $(\mathbb U,S)$ by $\mathbb U'$. 
Hence $\mathcal G'\to A$ is a $0$-definable connected groupoid relative to $A$ in $C_A(\mathcal G)$ and $\Aut_{\mathcal G_a'}(S_a)=\Aut(S_a/\mathbb U)$.

Next we establish the independence of the fibres.
The point is that the new structure in $C_A(\mathcal G)$ is just $\mathcal G'$ which has no morphisms between different fibres $S_a$ as the $\mathcal G'_a$ are its connected components.
But we give some details.
Since $S$ is the only new sort of $C_A(\mathcal G)$, and the fibres $S_a$ are $\mathbb U$-definable, the map $\Aut\big(C_A(\mathcal G)/\mathbb U\big)\to\prod_{a\in A}\Aut(S_a/\mathbb U)$ given by $\sigma\mapsto(\sigma|_{S_a}:a\in A)$ is an injective homomorphism.
To prove surjectivity we start with $\sigma_a\in \Aut(S_a/\mathbb U)$ for each $a\in A$, and define $\sigma:=\id_{\mathbb U}\cup \bigcup_{a\in A}\sigma_a$.
We need to show that $\sigma$ is an automorphism of $C_A(\mathcal G)$.
It suffices to show that $\sigma$ extends to an automorphism of $\mathbb U'=(\mathbb U, M_*, S,\dots)$.
Given $m\in M_*$ we define $\sigma(m)$ as follows: if $f_m\in\Hom_{\mathcal G'_a}(S_a, O_i)$ for some $i\in I_a$ then we define $\sigma(m)\in M_*$ to be such that $f_{\sigma(m)}=f_m\sigma_a^{-1}$.
Note that this makes sense as $\sigma_a\in \Aut(S_a/\mathbb U)=\Aut_{\mathcal G_a'}(S_a)$ and hence $f_m\sigma_a^{-1}\in\Hom_{\mathcal G'_a}(S_a, O_i)$.
If instead $f_m\in\Hom_{\mathcal G'_a}(O_i,S_a)$ then $\sigma(m)$ is such that $f_{\sigma(m)}=\sigma_af_m$.
Finally, if $f_m\in\Aut_{\mathcal G'_a}(S_a)$ then define $\sigma(m)$ so that $f_{\sigma(m)}=\sigma_af_m\sigma_a^{-1}$.
This extends $\sigma$ to a permutation of all the sorts of $\mathbb U'$.
The only new basic relation in $\mathbb U'$, besides $S\to A$ and $M_*\to A$ which $\sigma$ clearly preserves, is the ternary relation that exhibits elements of $M_*$ as encoding definable bijections involving the fibres of $S\to A$.
It is easy to check that we have defined $\sigma$ so that this relation is preserved.

Finally, we need to check that the language of $C_A(\mathcal G)$, say $L'$, is finitely generated over $L$.
This amounts to a relativisation of the proof Proposition~\ref{finite-language}, which we can effect using the independence of the fibres.
Since the fibres of $S\to A$ are internal to $\mathbb U$ which eliminates imaginaries, by compactness there exists a $0$-definable set $B$ in $C_A(\mathcal G)$ with a surjective $0$-definable map $B\to A$, and a $0$-definable map $F:S\times_AB\to X$ over $B$, for some sort $X$ of~$\mathbb U$, that is injective on the fibres over $B$.
So for $b\in B_a$, $F_b:S_a\to X$ is a $b$-definable embedding of $S_a$ into~$\mathbb U$.
There exists an $L'$-formula $\theta(v,w)$ such that $\models\theta(b,c)$ if and only if $b,c\in B_a$ for some $a\in A$, $F_b(S_a)=F_c(S_a)$ as subsets of $X$, and $F_c^{-1}\circ F_b:S_a\to S_a$ is in $\Aut(S_a/\mathbb U)$.
This is because $\Aut(S_a/\mathbb U)$ and its action on $S_a$ is $a$-definable; it is $\Aut_{\mathcal G_a'}(S_a)$.
By stable embeddedess (Lemma~\ref{sastable}) $\Aut(S_a/\mathbb U)$ is precisely the set of bijections $f:S_a\to S_a$ such that $f\cup\id_{\mathbb U}$ is a partial $L'$-elementary map.
It follows by compactness that there must be a finite sublanguage $\Sigma\subseteq L'$ such that whenever $b,c\in B_a$ and $F_b(S_a)=F_c(S_a)$, then $\displaystyle F_c^{-1}\circ F_b\cup\id_{\mathbb U}$ is a partial $L'$-elementary map if and only if it is a partial $\Sigma$-elementary map.
Extending $\Sigma$ we may also assume that $F$ and $B$ are defined by $\Sigma$-formulas.
Set $L_\circ':=L\cup\Sigma$.
We show that $L_\circ'$ witnesses the finite-generatedness of $L'$ over~$L$.
As in Proposition~\ref{finite-language}, this reduces to showing that $\Aut(\mathbb U'_\circ/\mathbb U)=\Aut(\mathbb U'/\mathbb U)$, where $\mathbb U'_\circ$ is the reduct of $\mathbb U'$ to~$L'_\circ$.
Suppose $\sigma\in \Aut(\mathbb U'_\circ/\mathbb U)$.
By the independence of fibres, it suffices to show that $\sigma|_{S_a}\in\Aut(S_a/\mathbb U)$ for each $a\in A$.
But fixing any $b\in B_a$, as we have seen a few times, we get $\sigma|_{S_a}=F_{\sigma(b)}^{-1}\circ F_b$.
In particular, $F_{\sigma(b)}^{-1}\circ F_b\cup\id_{\mathbb U}$ is a partial $L_{\circ}'$-elementary map, and hence a partial $L'$-elementary map by choice of $\Sigma\subseteq L'_\circ$.
So, $\sigma|_{S_a}\in\Aut(S_a/\mathbb U)$ as desired.
\end{proof}

When $A$ is a singleton $C_A(\mathcal G)$ coincides with the internal cove $C(\mathcal G)$ of~$\S$\ref{Cfunctor}.
Even when $A$ is infinite, $C_A(\mathcal G)$ may still be an internal cover.
For example, if $\mathcal G$ is such that between any two objects there is at most one morphism, then~$\mathcal G$ is essentially the $0$-definable equivalence relation $\operatorname{Iso}_{\mathcal G}$, and $C_A(\mathcal G)$ is part of $\mathbb U^{\operatorname{eq}}$.
Internality of $C_A(\mathcal G)$ corresponds to this happening for all but finitely many $a\in A$.

\begin{proposition}
The $1$-analysable cover $C_A(\mathcal G)$ is an internal cover if and only if for all but finitely many $a\in A$ the automorphism group of any (equivalently some) object of $\mathcal G_a$ is trivial.
\end{proposition}

\begin{proof}
For any cover $\mathbb U'=(\mathbb U,S,\dots)$, in a finitely generated language, internality is equivalent to the cardinality of $\Aut(\mathbb U'/\mathbb U)$ being bounded by $|\mathbb U|$.
This characterisation seems to be well known; for example, it is suggested without proof in~\cite[Remark~1.2]{hrushovski}.
The left-to-right direction is because internality implies that $\Aut(\mathbb U'/\mathbb U)$ is interpretable in $\mathbb U'$ and hence cannot have cardinality greater than that of the universe.
For the converse one observes that if $S$ is not internal to $\mathbb U$ then for any partial elementary map $f$ with small domain $A\subseteq S$, since $S\not\subseteq\dcl(\mathbb U,A)$, $f\cup\id_{\mathbb U}$ can be extended in two distinct ways to a partial elementary map on $A\cup\{a\}$ for some $a\in S$.
We can thus build, back-and-forthly, a binary tree of partial elementary maps whose paths give rise to $2^{|\mathbb U|}$-many elements of $\Aut(\mathbb U'/\mathbb U)$.

Now, by Theorem~\ref{relative-cover-construction} we have that
$$\displaystyle \Aut\big(\mathbb C_A(\mathcal G)/\mathbb U\big)=\prod_{a\in A}\Aut(S_a/\mathbb U)=\prod_{a\in A}\Aut_{\mathcal G_a'}(S_a).$$
The set of $a\in A$ such that $\Aut_{\mathcal G_a'}(S_a)$ is not trivial is definable, and hence is either finite or of the cardinality of the universe.
In the former case $|\Aut\big(\mathbb C_A(\mathcal G)/\mathbb U\big)|\leq|\mathbb U|$ and in the latter $|\Aut\big(\mathbb C_A(\mathcal G)/\mathbb U\big)|=2^{|\mathbb U|}$.
\end{proof}

Theorems~\ref{relative-groupoid-construction} and~\ref{relative-cover-construction} relativise Hrushovki's correspondence between definable connected groupoids and internal generalised imaginary sorts.
In fact the $C_A$ and $G_A$ given by these theorems are functorial with respect to the following natural notions of morphism:
Given definable families of connected groupoids, $\mathcal G_1\to A$ and $\mathcal G_2\to A$, a {\em morphism} is simply a definable family of morphisms of definable connected groupoids $h_a:{\mathcal G_1}_a\to {\mathcal G_2}_a$ parameterised by $A$.
Given $1$-analysable covers $\mathbb U_1=(\mathbb U,S_1,\dots)$ and $\mathbb U_2=(\mathbb U,S_2,\dots)$ over $A$, a {\em morphism} between them is just a morphism of covers $g:S_1\to S_2$ such that
\begin{center}
\begin{tikzcd}
S_1\ar[rr, "g"]\ar[dr]&&S_2\ar[dl]\\
&A
\end{tikzcd}
\end{center}
commutes.
The arguments of~$\S$\ref{section-functorial} relativise so that $C_A$ and $G_A$ are functors witnessing the equivalence of the following two categories:
\begin{itemize}
\item
finitely faithful $0$-definable connected groupoids relative to~$A$, and
\item
$1$-analysable covers over $A$ with independent fibres.
\end{itemize}
We leave it to the reader to verify the details.

For a $1$-analysable cover $S\to A$ to have independent fibres is a rather restrictive condition.
One can think of it as asking that the cover $\mathbb U'=(\mathbb U, S,\dots)$ be as far from being internal to $\mathbb U$ as is consistent with the fibres of $S\to A$ being $\mathbb U$-internal.
For example, in differentially closed fields of characteristic zero, the best known non-internal $1$-analysable cover of the field of constants $C$ is the subgroup $S$ of the multiplicative group defined by the second order equation $\delta\big(\frac{\delta x}{x}\big)=0$, with $S\to C$ being given by the logarithmic derivative $x\mapsto\frac{\delta x}{x}$.
But the fibres of $S\to C$ are not independent: $S$ admits a definable group structure (field multiplication) that intertwines the fibres.
Ruizhang Jin has, however, communicated to us the following positive example in differentially closed fields, details of which will appear in his PhD thesis.

\begin{example}[Jin]
Suppose $\mathbb M=(K,0,1,+,-,\times,-,\delta)$ is a saturated differentially closed field in characteristic zero, $\mathbb U=(C,0,1,+,-,\times)$ is its field of constants, and $t\in K$ is a fixed element with $\delta t=1$ that we will use as a parameter.
Consider the $t$-definable set,
$$S:=\{x\in K:\frac{\delta x}{x}=\frac{1}{(t+c)^2}\text{ for some }c\in\mathbb U\}$$
and the $t$-definable map $\pi:S\to C$ given by $\pi(x)=\delta\big(\frac{x}{2\delta x})-t$.
Then the fibre above $c\in C$ is $S_c=\{x:\frac{\delta x}{x}=\frac{1}{(t+c)^2}\}$, which is a multiplicative translate of the multiplicative group of constants, and hence is internal to $\mathbb U$.
But, as Jin has worked out, any finite set of elements of $S$ coming from different fibres will be $\operatorname{acl}$-independent over $C(t)$.
This can be used to show that the $t$-definable structure induced on $(C,S)$ is a $1$-analysable cover over $C$ with independent fibres.
\end{example}

The question remains as to what definable object intrinsic to $\mathbb U$ captures the data of a $1$-analysable cover $\mathbb U'=(\mathbb U, S,...)$ over $A$ whose fibres may not be independent.
Probably one should make the additional assumption that $(\mathbb U,S_a)$ is stably embedded in $\mathbb U'$ for each $a\in A$.
Hrushovski~\cite{hrushovski} suggests ``definable simplicial groupoids".
It would be of interest to realise this suggestion.

\bigskip

\bibliographystyle{plain}

%\bibliography{ref}

\end{document}